\documentclass[11pt,reqno]{amsart}
\usepackage{amsmath,amsfonts,amsthm,amssymb,graphicx}

\theoremstyle{plain}
\newtheorem{theorem}{Theorem}[section]
\newtheorem{lemma}{Lemma}[section]
\newtheorem{proposition}{Proposition}[section]

\theoremstyle{definition}
\newtheorem{definition}{Definition}[section]

\theoremstyle{remark}
\newtheorem{remark}{Remark}[section]

\theoremstyle{example}
\newtheorem{example}{Example}[section]
\theoremstyle{corollary}
\newtheorem{corollary}{Corollary}[section]

\numberwithin{equation}{section}

\newcommand{\R}{\mathbb{R}}

\begin{document}
\title[Navier-Stokes Equations
with the Navier Boundary Condition]{A Study of the Navier-Stokes
Equations with the Kinematic and Navier Boundary Conditions}
\author{Gui-Qiang Chen \and Zhongmin Qian}
\address{G.-Q. Chen,  School of Mathematical Sciences, Fudan University,
 Shanghai 200433, China; Department of Mathematics, Northwestern University,
         Evanston, IL 60208-2730, USA}
\email{\tt gqchen@math.northwestern.edu}
\address{Z. Qian, Mathematical Institute, University of Oxford, 24-29 St Giles, Oxford OX1 3LB, UK}
\email{\tt qianz@maths.ox.ac.uk}

\date{\today}
\thanks{}

\begin{abstract}
We study the initial-boundary value problem of the Navier-Stokes
equations for incompressible fluids in a domain in $\R^3$
with compact and smooth boundary, subject to the kinematic and
Navier boundary conditions.
We first reformulate the Navier boundary condition in terms of the
vorticity, which is motivated by the Hodge theory on manifolds with
boundary from the viewpoint of differential geometry, and establish
basic elliptic estimates for vector fields subject to the kinematic
and Navier boundary conditions. Then we develop a spectral theory of
the Stokes operator acting on divergence-free vector fields on a
domain with the kinematic and Navier boundary conditions.
Finally, we employ the spectral theory and the necessary estimates
to construct the Galerkin approximate solutions and establish their
convergence to global weak solutions, as well as local strong
solutions, of the initial-boundary problem. Furthermore, we show as
a corollary that,
when the slip length tends to zero, the weak solutions constructed
converge to a solution to the incompressible Navier-Stokes equations
subject to the no-slip boundary condition for almost all time. The
inviscid limit of the strong solutions to the unique solutions of
the initial-boundary value problem with the slip boundary condition
for the Euler equations is also established.
\end{abstract}
\maketitle

\section{Introduction}

We are concerned with solutions of the initial-boundary value
problem of the Navier-Stokes equations for incompressible fluids in
a general domain in $\R^3$ with compact and smooth boundary, subject
to the kinematic boundary condition (i.e. the slip condition) and
the Navier boundary condition. The incompressible fluid flows
are governed by the Navier-Stokes equations:
\begin{equation}
\partial_t u+u\cdot \nabla u=\mu \Delta u-\nabla p, \qquad
\nabla\cdot u=0,  \qquad\qquad x\in\Omega\subset \R^3,
\label{23-n-1}
\end{equation}
where $u$ represents the Eulerian velocity vector field of the fluid
flow,
$p$ is a scalar pressure function
(up to a function of time $t$) which maintains the incompressibility
of the fluid, $\mu$ is the kinematic viscosity, and the density has
been renormalized as one in this setting. As a system of partial
differential equations, $u$ and $p$ are unknown functions, and a
solution $u$ of (\ref{23-n-1}) determines $p$ uniquely up to a
function depending only on $t$. The initial condition for the fluid
flow is
\begin{equation}\label{initial-1}
u|_{t=0}=u_0(x).
\end{equation}

If $\Omega\subset\R^3$ has a non-empty boundary
$\Gamma=\partial\Omega$, then system (\ref{23-n-1}) must be
supplemented with boundary conditions on $u$ in order to be
well-posed. In fluid dynamics, if the rigid surface $\Gamma$ is at
rest, the kinematic and no-slip conditions are often imposed. The
kinematic condition means that the normal component of the velocity
vanishes, that is, the velocity $u$ is tangent to the boundary
$\Gamma$:
\begin{equation}\label{kinematic-1}
u^{\bot}|_{\Gamma}=0,
\end{equation}
while the no-slip condition
demands for the coincidence of the tangent component of the fluid
velocity with that of the boundary $\Gamma$. These two boundary
conditions lead to the Dirichlet boundary problem associated with
the Navier-Stokes equations. There has been a large literature for
the Navier-Stokes equations subject to the Dirichlet boundary
condition; see
\cite{hopf1,kreiss-lorenz1,lad3,leray2,leray3,leray1,serrin63,temam1,wahl1}
and the references cited therein. The fundamental problem of the
global (in time) existence and uniqueness of a strong solution
remains open; however, the Dirichlet boundary problem of the
Navier-Stokes equations
is well-posed at least for a small time, or for small data globally
in time.

However, the usual no-slip assumption does not always match with the
experimental results. Navier \cite{navier-1} first proposed the
slip-with-friction boundary condition, that is, the Navier boundary
condition: The tangent part of the velocity $u$ is proportional to
that of the normal vector field of the stress tensor
with proportional constant $\zeta>0$, which is called the slip
length (see (\ref {23-11-0}) below). In the recent years, the Navier
boundary condition has been received much attention, especially when
fluids
with larger Reynolds number or fluids past a rigid surface with
considerable speeds for which the curvature effect becomes apparent
(cf. \cite{Einzel-Panzer-Liu,IRS,JM1,JM2,lauga-etc,zhu-Garnick}
and the references cited therein). Such boundary conditions can be
induced by effects of free capillary boundaries, a perforated
boundary, or an exterior electric field (cf.
\cite{APV,Bansch,BJ,CDE,QWS,Schwartz}). In particular, this friction
boundary condition was rigorously justified as the effective
boundary condition for flows over rough boundary; see
\cite{JM-1,JM1}. Thus, it becomes important to analyze solutions to
the equations for such fluids subject to the Navier boundary
condition.

The rigorous mathematical analysis of the Navier-Stokes equations
with the kinematic and Navier boundary conditions may date back the
work by Solonnikov-{\v{S}}{\v{c}}adilov \cite{ss} for the stationary
linearized Navier-Stokes system under the boundary condition that
the tangent part of the normal vector field of the stress tensor is
zero. The existence of weak solutions and regularity for the
stationary Navier-Stokes equations with the kinematic and Navier
boundary conditions was only recently obtained by Beir\~{a}o da
Veigt \cite{BV} for the half-space. In a two-dimensional, simply
connected, bounded domain, the well-posedness problem has been
rigorously established by Yodovich \cite{Yo}. See also
Clopeau, Mikeli\'{c}, and Robert \cite{CMR} and Lopes Filho,
Nussenzveig Lopes and Planas \cite{LLP} for the vanishing viscosity
limit, and Mucha \cite{Mucha} under some geometrical constraints on
the shape of the domain. These two-dimensional results are based on
the fact that the vorticity is scalar and satisfies the maximum
principle. However, in the three-dimensional case,
the standard maximum principle for the vorticity fails, so that the
techniques employed in the two-dimensional case can not be directly
extended to this case. Furthermore, the Navier boundary condition
causes additional difficulties in developing apriori estimates which
require to be compatible with the nonlinear convection term. The
main purpose of this paper is to develop a general approach to
establish the well-posedness, the no-slip limit, as well as the
inviscid limit for the initial-boundary value problem for the
Navier-Stokes equations in a general domain $\Omega\subset\R^3$,
subject to the Navier boundary
condition, together with the kinematic condition \eqref{kinematic-1}
and the initial condition \eqref{initial-1}.

By careful local computations in Section 2,
we first reformulate the Navier boundary condition in terms of the
vorticity $\omega=\nabla\times u$ (see Proposition \ref{navier-c}
below). This is motivated by the Hodge theory on manifolds with
boundary, since the kinematic and vorticity conditions are the
natural boundary conditions for the Hodge theory from the viewpoint
of differential geometry. Indeed, the Navier boundary condition
under the kinematic condition \eqref{kinematic-1} is equivalent to
the condition that the tangent portion of the vorticity $\omega$ is
of the following form:
\begin{equation}
\left. \omega^{\Vert}\right|_{\Gamma} =-\frac{1}{\zeta}(\ast
u)+2\big(\ast\pi(u)\big),  \label{nav-f}
\end{equation}
where $\pi =(\pi_{ij})$ is the curvature of the boundary $\Gamma$,
that is, the second fundamental form. The form $\pi$ is identified
with the self-adjoint operator which sends a tangent vector
$(\xi^{1}, \xi^{2})$ on the boundary surface $\Gamma$ to the tangent
vector $(\sum_{j}\pi_{1j}\xi^{j},\sum_{j}\pi_{2j}\xi^{j})$, the
operator $\ast$ is the Hodge star operator which rotates (towards
the interior of the domain $\Omega$) a tangent vector $(\xi^{1},
\xi^{2})$ by 90$^{0}$ degree. The Navier boundary condition in form
(\ref{nav-f}) has appealing physical interpretation:
The vorticity on the boundary $\Gamma$ is created mainly from the
slip of the fluid and the curvature of the boundary $\Gamma$
(see \cite{lauga-etc} for more information and further references on
the slip length for different fluid media). In particular, the
effect of the curvature becomes significant when the curvature of
the boundary $\Gamma$ becomes large (comparable to the reciprocal of
the slip length).

In Section 3, we establish some basic elliptic estimates for vector
fields subject to the kinematic and Navier boundary conditions. In
particular, we establish some $L^{2}$-estimates which are uniform in
the slip length by a careful analysis of several boundary integrals.

In Section 4, we develop a spectral theory of the Stokes operator
acting on divergence-free vector fields in a general domain $\Omega$
subject to the kinematic and Navier boundary conditions.
 We establish several fundamental estimates for
the symmetric form defined by the Stokes operator. Besides the
difficulties caused by the divergence-free condition on the vector
fields,
the Navier boundary condition causes additional difficulties in
developing \emph{apriori} estimates for the Galerkin approximations
to solutions of the Navier-Stokes equations. To overcome these
difficulties, we establish an estimate for the third derivatives of
the vector fields satisfying the Navier boundary condition (Theorem
\ref{le-23-6} and Corollary \ref{co-23-31})
and a uniform gradient estimate for the Galerkin
approximations (Theorem \ref{th-n4}).
Then, in Section 5, we employ the spectral theory and all the
estimates established in Sections 3--4 to construct the Galerkin
approximate solutions and establish the global existence of weak
solutions and the local existence of strong solutions for the
initial-boundary problem \eqref{23-n-1}--\eqref{nav-f}.
Furthermore, we show as a corollary that, for any weak solution
$u_\zeta(t,x)$ corresponding to the slip length $\zeta$ to problem
\eqref{23-n-1}--\eqref{nav-f} constructed in Theorem {\rm
\ref{w-th-01}}, when $\zeta\to 0$, there exists a subsequence (still
denoted) $u_\zeta(t,x)$ converging to $u(t,x)$ such that $u(t,x)$ is
a solution to \eqref{23-n-1} subject to the no-slip condition for
almost all time $t$. On the other hand, when $\zeta\to \infty$,
there also exists a subsequence (still denoted) $u_\zeta(t,x)$
converging to $u(t,x)$ such that $u(t,x)$ is a solution to
\eqref{23-n-1} subject to the complete slip boundary condition:
$$
\left. \omega^{\Vert}\right|_{\Gamma} =2\big(\ast\pi(u)\big),
$$
in the weak sense.
Such a nonhomogeneous vorticity boundary problem has been carefully
investigated in \cite{C-Q1}.

Finally, in Section 6, we study the inviscid limit and establish the
$L^2$--convergence of the strong solutions of problem
\eqref{23-n-1}--\eqref{nav-f} to the unique smooth solution of the
initial-boundary value problem with the slip boundary condition for
the Euler equations for incompressible fluid flows.

\section{The Navier boundary condition}

In this section we reformulate the Navier boundary condition in
terms of the vorticity
for every slip length $\zeta>0$ and introduce several notions and
notations which are used throughout the paper.

For simplicity, we use the conventional notation that the repeated
indices in a formula are understood to be summed up from $1$ to $3$
unless confusion may occur. Furthermore, we use a universal constant
$C>0$ that is independent of the slip length $\zeta>0$, and a
universal constant $M>0$ that may depend on $\zeta$ among others,
which may be different at each occurrence.

We use the same notation for both scalar functions and
vector fields in the $L^{p}$-space and Sobolev spaces
$W^{k,p}(\Omega)$ ($H^k(\Omega)$ if $p=2$).  Denote $\|T\|_{p}$ as
the $L^{p}$-norm of the length $|T|$ of $T$ on $\Omega$ with respect
to the Lebesgue measure, and $\|T\|_{L^{p}(\Gamma)}$ as the
$L^{p}$-norm of the vector field $T$ on the boundary $\Gamma$ with
respect to the induced surface-area measure on $\Gamma$. That is,
\begin{equation*}
\|T\|_{p}=\Big(\int_{\Omega}|T(x)|^{p}\, dx\Big)^{\frac{1}{p}},
\qquad \|T\|_{L^{p}(\Gamma)}=\Big(\int_{\Gamma}|T(x)|^{p}
\,d\mathcal{H}^2(x)\Big)^{\frac{1}{p}},
\end{equation*}
where $dx$ is the usual Lebesgue measure on $\mathbb{R}^{3}$ and
$\mathcal{H}^2$ is the two-dimensional Hausdorff measure (i.e. the
surface area measure) on $\Gamma$. From now on, $dx$ and
$d\mathcal{H}^2(x)$ in the integrals  will be suppressed, unless
confusion may arise. For a vector field $T$ on $\Omega$,
\begin{equation*}
\|T\|_{W^{k,p}}=\Big(\sum_{j=0}^k\int_{\Omega}|\nabla^{j}T|^{p}\Big)^{\frac{1}{p}},
\end{equation*}
where $\nabla^{j}T$ is the $j$-th derivative of $T$. See
\cite{Adams-Fournier03,gri1} for the details.

\smallskip
For the bounded domain $\Omega\subset \R^3$
with a boundary $\Gamma=\partial\Omega$ that is smooth, compact, and
oriented, unless otherwise specified, we carry out local
computations on the boundary
in a moving frame compatible to $\Gamma$. More precisely, if $\nu$
is the unit normal to $\Gamma$ pointing outwards with respect to
$\Omega$, by a moving frame we mean any local orthonormal basis
$(e_{1}, e_2, e_{3})$
of the tangent space $T\Omega$ such that $e_{3}=\nu$ when restricted
to $\Gamma$. If $u=\sum_{j=1}^3u^{j}e_{j}$ is a vector field on
$\Omega$, then, restricted to the boundary surface $\Gamma$,
$u^{\Vert}=\sum_{j=1,2}u^{j}e_{j}$ (resp. $u^{\bot }=u^{3}\nu$)
denotes its tangent part (resp. normal part). The Christoffel
symbols $\Gamma_{ij}^{l}$ are determined by the directional
derivatives $\nabla_{i}e_{j}=\Gamma_{ij}^{k}e_{k}$, where
$\nabla_{i}$ is the directional derivative in the direction $e_{i}$.

The tensor $(\pi_{ij})_{1\le i,j\leq 2}$, where
$\pi_{ij}=-\Gamma_{ij}^{3}$ for $i,j=1,2$, is a symmetric tensor on
$\Gamma$, which is the second fundamental form, denoted by $\pi$.
That is,
\begin{equation*}
\pi(u^{\Vert}, v^{\Vert}):=\sum_{i,j=1,2}\pi_{ij}u^{i}v^{j}\qquad
\text{for any } u^{\Vert}, v^{\Vert}\in T\Gamma \text{.}
\end{equation*}
We say that $\pi$ is bounded above (resp. below) by a constant
$\lambda$ if two eigenvalues (which are functions on $\Gamma$) are
bounded above (resp. below) by $\lambda$. We will also identify
$\pi$ with the linear transformation $(h_{ij})$: If
$u^{\Vert}=\sum_{j=1,2}u^{j}e_{j}$ is tangent to $\Gamma$, then
\begin{equation*}
\pi (u^{\Vert}):=\sum_{j=1,2}\pi (u^{\Vert})^{j}e_{j}=\sum_{j=1,2}
{\pi}_{ij}u^{i}e_{j},
\end{equation*}
with $\langle \pi(u^{\Vert}), v^{\Vert}\rangle =\pi(u^{\Vert},
v^{\Vert})$, and $H=\sum_{j=1,2}{\pi}_{jj}$ is the mean curvature. We
refer to \cite{Palais-Terng,derham1} for further facts and notations
in differential geometry used in this paper.

The boundary surface $\Gamma\subset\R^{3}$ has a natural induced
metric and hence a natural notion of directional derivatives, the
L\'{e}vi-Civita connection, denoted by $\nabla^{\Gamma}$. The
following formulas will be useful in treating with integrals on the
boundary $\Gamma$. Let $u\in H^{2}(\Omega)$ be a vector field on
$\Omega$.
Then, on $\Gamma$,
\begin{eqnarray}
\langle u\cdot\nabla u, \nu\rangle&=&-\pi
(u^{\Vert},u^{\Vert})-H|u^{\bot}|^{2}+\langle u,\nu\rangle\nabla\cdot u  \notag \\
&&+2\langle u^{\Vert}, \nabla^{\Gamma}\langle u,\nu \rangle \rangle
-\nabla^{\Gamma}\cdot\big(\langle u,\nu \rangle u^{\Vert}\big),
\label{june21-045}
\end{eqnarray}
and
\begin{equation}
\frac{1}{2}\partial_\nu (|u|^{2})=\langle u\times (\nabla \times
u),\nu \rangle + \langle u\cdot\nabla u,\nu \rangle.
\label{june-063}
\end{equation}
The first formula (\ref{june21-045}) may be verified by means of the
moving frame method (see \cite{C-Q1} for the details). The second
follows from the vector identity:
\begin{equation}
\frac{1}{2}\nabla |u|^{2}=u\times(\nabla \times u)+u \cdot\nabla u.
\label{v-24-1}
\end{equation}

Let $f$ be a scalar function on $\Omega $. Then, as a special case
of (\ref {june-063}),
\begin{eqnarray}
\partial_\nu\big(|\nabla f|^{2}\big)&=&-2\pi(\nabla f^{\Vert},\nabla f^{\Vert})
-2H\left|\partial_\nu f\right|^{2}+
\partial_\nu f\Delta f  \notag \\
&&+\, 4\langle \nabla f^{\Vert},\nabla^{\Gamma}(\partial_\nu
f)\rangle -2\nabla ^{\Gamma}\cdot\big(\partial_\nu
f\nabla^{\Gamma}f^{\Vert}\big)\text{.}  \label{05-01}
\end{eqnarray}

The connecting condition over an interface $\Gamma $ of a fluid is
expressed as \emph{the Navier boundary condition}; see
Einzel-Panzer-Liu \cite{Einzel-Panzer-Liu} for its physical
interpretation. This condition may be written in a moving frame
compatible to $\Gamma$ as follows:
\begin{equation}
u^{k}=-\zeta\left(\nabla_{3}u^{k}+\nabla_{k}u^{3}\right)\qquad
\text{on}\,\, \Gamma \qquad\text{for $k=1,2$}, \label{23-11-0}
\end{equation}
where $\zeta$ is the slip length that is a positive scalar function
on $\Gamma$ depending only on the nature of the fluid and the
material of the rigid boundary. In order to write down
\eqref{23-11-0}
in a global form in terms of the vorticity and the curvature of
$\Gamma$, we recall that the Hodge operator $\ast$ sends a vector
field $(v^{1}, v^{2})$ on the surface $\Gamma$ to
$\ast(v^1,v^2):=(-v^{2}, v^{1})$. The effect of the Hodge operator
$\ast$ is to rotate a vector on $\Gamma $ by 90$^{0}$ degree with
respect to the normal vector pointing the interior of $\Omega$. The
Hodge operator $\ast$ is independent of the choice of a moving frame
on $\Gamma$ and may be defined via the identity:
\begin{equation}
\langle w\times(\ast u^{\Vert}), \nu\rangle =\langle u^{\Vert},
w^{\Vert}\rangle \qquad \text{on}\,\, \Gamma \label{id-30-01}
\end{equation}
for any vector fields $u$ and $w$.

\begin{proposition}\label{navier-c}
Let $u\in C^1(\Omega)$. Then $u$ satisfies the Navier boundary
condition on $\Gamma$ if and only if
\begin{equation}
\left(\nabla \times u\right)^{\Vert}\Big|_{\Gamma}
=-\frac{1}{\zeta}(\ast u^{\Vert})-2\big(\ast\nabla^{\Gamma}\langle
u,\nu \rangle\big) +2\big(\ast\pi(u^{\Vert})\big).
\label{n01-teq1}
\end{equation}
In terms of the components in a moving frame compatible to the
boundary, condition \eqref{n01-teq1} takes the following forms:
\begin{eqnarray*}
&&\left(\nabla\times u\right)^{1}
=\frac{1}{\zeta}u^{2}+2\nabla_{2}\langle u,\nu \rangle -2\sum_{j=1,2}\pi_{j2}u^{j},\\
&&\left( \nabla \times
u\right)^{2}=-\frac{1}{\zeta}u^{1}-2\nabla_{1}\langle u,\nu \rangle
+2\sum_{j=1,2}\pi_{j1}u^{j}.
\end{eqnarray*}
In particular, if $u$ satisfies the kinematic condition
\eqref{kinematic-1},
then the Navier condition \eqref{23-11-0} is equivalent to
\eqref{nav-f}.
\end{proposition}

\begin{proof}
It suffices to show the results in a moving frame compatible to the
boundary surface $\Gamma$. Then, for $k=1,2$,
\begin{eqnarray*}
\nabla_{k}u^{3}=e_{k}(u^{3})+\sum_{j=1,2}\Gamma_{kj}^{3}u^{j}
=e_{k}\langle u,\nu \rangle -\sum_{j=1,2} \pi_{jk}u^{j}.
\end{eqnarray*}
Therefore, along the boundary surface $\Gamma$,
\begin{equation*}
\nabla_{3}u^{k}+\nabla_{k}u^{3}=\varepsilon_{3ka}\omega^{a}
 +2e_{k}\langle u,\nu \rangle -2\sum_{j=1,2}\pi_{jk}u^{j}
 \qquad\,\,\, \text{for any $k=1,2$},
\end{equation*}
so that
\begin{equation*}
\left.
\omega^{a}\right|_{\Gamma}=-\frac{1}{\zeta}\varepsilon_{3ka}u^{k}
-2\varepsilon_{3ka}e_{k}\langle u,\nu \rangle
+2\varepsilon_{3ka}\sum_{j=1,2}\pi_{jk}u^{j} \qquad\,\,\, \text{for
$a=1,2$},
\end{equation*}
where $\varepsilon_{ijk}$ is the Kronecker symbols. This completes
the proof.
\end{proof}

\begin{definition}
Let $\zeta >0$ be a constant. Then a vector field $u$ on $\Omega$ is
said to satisfy the Navier's $\zeta$-condition if \eqref{nav-f}
holds,
which is equivalent to
\begin{equation}
\left(\nabla\times u\right)^{1}=\frac{1}{\zeta}u^{2}
-2\sum_{j=1,2}\pi_{j2}u^{j}, \qquad \left(\nabla\times
u\right)^{2}=-\frac{1}{\zeta}u^{1} +2\sum_{j=1,2}\pi_{j1}u^{j}
\label{cceq01}
\end{equation}
in a moving frame compatible to the boundary surface $\Gamma$.
\end{definition}

In this paper we study the initial-boundary problem
\eqref{23-n-1}--\eqref{nav-f} for the Navier-Stokes equations
with fixed constants $\mu >0$ and $\zeta>0$.

\section{Elliptic estimates for vector fields}

The fundamental estimates in the standard elliptic theory state
that, for any function $f$ on $\Omega$ subject to a certain boundary
condition (Dirichlet or Neumann),
$\|f\|_{H^{2}}$ is dominated by the $L^{2}$-norm of its Laplacian
together with its $L^{2}$-norm:
\begin{equation*}
\|f\|_{H^{2}}\leq C(\|\Delta f\|_{2}+\|f\|_{2})
\end{equation*}
for some constant $C$ depending only on $\Omega$. A correct boundary
condition here plays an essential role, and the previous estimate
can not be true without a proper boundary condition. A version of
elliptic estimates for vector fields has been established in
\cite{agmon1965} (also see \cite{morrey1966}) for vector fields
satisfying the Dirichlet or Neumann condition. If $u$ is a vector
field in $H^{2}(\Omega)$ such that $\left.
u^{\bot}\right|_{\Gamma_{1}}=0$ and $\left. u\right|_{\Gamma_{2}}=0$
for $\Gamma =\Gamma_{1}\cup\Gamma_{2}$, then
\begin{equation}
\|u\|_{H^{1}(\Omega)}\leq C\left(\|\nabla\times u\|_{2}
+\|\nabla\cdot u\|_{2}+ \|u\|_{2}\right), \label{07june06-01}
\end{equation}
which is a special case of a general result in \cite{agmon1965}.

For our problem, we need to develop the $L^{2}$-estimates for
divergence-free vector fields that satisfy the kinematic condition
\eqref{kinematic-1} and Navier's $\zeta$-condition \eqref{nav-f} in
a general domain $\Omega$. To our knowledge, these estimates are not
covered in the previous literature, although they can be considered
as a part of the standard elliptic theory.

We begin with an elementary lemma which implies (\ref{07june06-01})
and may be verified by means of integration by parts.

\begin{lemma}\label{le-66-1}
Let $u\in H^2(\Omega)$ be a vector field on $\Omega$. Then
\begin{equation}
\int_{\Omega}\langle \Delta u,u\rangle =-\|\nabla
u\|_{2}^{2}+\frac{1}{2}\int_{\Gamma} \partial_\nu(|u|^{2})
\label{07june03-03}
\end{equation}
and
\begin{eqnarray}
\int_{\Omega}|\nabla u|^{2}=\|\nabla\times u\|_{2}^{2}
+\|\nabla\cdot u\|_{2}^{2}-\int_{\Gamma}\left(\nabla\cdot u\right)
\langle u, \nu\rangle +\int_{\Gamma}\langle u\cdot \nabla u,\nu
\rangle.
\label{t-d-01}
\end{eqnarray}
In particular, if $\left. u^{\bot}\right|_{\Gamma}=0$, then
\begin{equation}
\|\nabla u\|_{2}^{2}=\|\nabla\times u\|_{2}^{2}+\|\nabla\cdot
u\|_{2}^{2}-\int_{\Gamma}\pi \left(u, u\right).
\label{est-june26-01a}
\end{equation}
\end{lemma}

\begin{lemma}\label{le-66-02}
If $g$ is a smooth function on $\Omega$ (up to the boundary
$\Gamma$), then
\begin{eqnarray}
\|\nabla^2g\|_{2}^{2}=\|\Delta g\|_{2}^{2}-\int_{\Gamma}\pi ((\nabla
g)^{\Vert}, (\nabla g)^{\Vert}) -\int_{\Gamma}H\left|\partial_\nu
g\right|^{2} +2\int_{\Gamma}\langle \nabla^{\Gamma}g,
\nabla^{\Gamma}(\partial_\nu g)\rangle. \label{4-30-11}
\end{eqnarray}
\end{lemma}

This elementary fact can be  proved by using integration by parts
and the Bochner identity:
\begin{equation}\label{Bochner}
|\nabla^2 g|^{2}=\frac{1}{2}\Delta |\nabla g |^{2}-\langle \nabla
\Delta g, \nabla g \rangle.
\end{equation}

\smallskip
Now we are in a position to prove our first main estimate, an
elliptic estimate, for vector fields satisfying
\eqref{kinematic-1}--\eqref{nav-f}.

\begin{theorem}
\label{th-66-1}There exists a constant $C$ depending only on
$\Omega$ such that
\begin{equation}
\|\nabla^2 u\|_{2}^{2}
+\frac{1}{\zeta}\|\nabla^{\Gamma}u\|_{L^{2}(\Gamma)}^{2} \leq
C\left( \|\Delta u\|_{2}^{2}+\|u\|_{H^{1}}^{2}\right)
\label{in-5-04}
\end{equation}
for any vector field $u$ satisfying \eqref{kinematic-1}--\eqref{nav-f}.
\end{theorem}

\begin{proof}
Under an orthonormal frame, we have
\begin{eqnarray}
\|\nabla^2u\|_{2}^{2}
&=&\sum_{k=1}^3\int_{\Omega}(\Delta u^{k})^{2}
-\int_{\Gamma}\sum_{i,j=1,2}\pi_{ij}(\nabla_{i}u^{k})(\nabla_{j}u^{k}) \notag\\
&&-\int_{\Gamma}H\sum_{k=1}^3\big|\partial_\nu u^{k}\big|^{2}
  +2\int_{\Gamma}\langle \nabla^{\Gamma}u^{k},\nabla^{\Gamma}(\partial_\nu u^{k})\rangle
  \notag \\
&=&\|\Delta u\|_{2}^{2}-\int_{\Gamma}\pi((\nabla u^{k})^{\Vert},
(\nabla u^{k})^{\Vert}) -\int_{\Gamma}H\sum_{k=1}^3\langle\nabla
u^{k},
\nu\rangle^{2}\notag\\
&& +2
\int_{\Gamma}\langle \nabla^{\Gamma}u^{k},
\nabla^{\Gamma}\langle \nabla u^{k},\nu\rangle \rangle,
\label{5-1-01}
\end{eqnarray}
where the second equality follows from (\ref{4-30-11}) applying to
$g=u^{k}$. The second and third boundary integrals can be dominated
by $\int_{\Gamma}|\nabla u|^{2}$. Thus, we have to handle the last
boundary integral, where we use the Navier boundary condition
\eqref{nav-f}. Working in a frame compatible to $\Gamma $, since
$\left. u^{\bot}\right|_{\Gamma}=0$ so that
$\nabla^{\Gamma}u^{3}=0$, then
\begin{eqnarray*}
I=\sum_{k=1}^3\int_{\Gamma}
\langle\nabla^{\Gamma}u^{k},\nabla^{\Gamma}(\partial_\nu
u^{k})\rangle =\sum_{k=1,2}\int_{\Gamma}\langle
\nabla^{\Gamma}u^{k}, \nabla^{\Gamma}(\partial_\nu u^{k})\rangle.
\end{eqnarray*}
On $\Gamma$, $u^1$ and $u^2$ are the tangent components of $u$, and
$u^{3}=0$,
\begin{equation*}
\partial_\nu u^{k}=e_{3}(u^{k})=\nabla_{3}u^{k}-\sum_{j=1,2}u^{j}\Gamma_{3j}^{k},
\end{equation*}
and $\nabla_{k}u^{3}=-\sum_{i=1,2}u^{i}\pi_{ki}$, $k=1,2$. For
$\omega=\nabla\times u$,
$\nabla_{3}u^{k}-\nabla_{k}u^{3}=\varepsilon_{3kj}\omega^{j}$ so
that, for $k=1,2$,
\begin{eqnarray}
\partial_\nu u^{k}=\varepsilon_{3kj}\omega^{j}+\nabla_{k}u^{3}
-\sum_{j=1,2}u^{j}\Gamma_{3j}^{k}
=\varepsilon_{3kj}\omega^{j}-\sum_{i=1,2}u^{i}\pi_{ki}
-\sum_{j=1,2}u^{j}\Gamma_{3j}^{k}. \label{es-5-02}
\end{eqnarray}
According to the Navier's $\zeta$--condition (\ref{nav-f}) (also see
the proof of Proposition \ref{navier-c}):
\begin{equation*}
\omega^{j}=-\frac{1}{\zeta}\varepsilon_{3aj}u^{a}
+2\varepsilon_{3aj}\sum_{b=1,2}h_{ba}u^{b}.
\end{equation*}
Substitution it into (\ref{es-5-02}) yields
\begin{eqnarray}
\partial_\nu u^{k}&=&-\frac{1}{\zeta}\varepsilon_{3aj}\varepsilon_{3kj}u^{a}
 +2\varepsilon_{3kj}\varepsilon_{3aj}\sum_{b=1,2}h_{ba}u^{b}
 -\sum_{i=1,2}u^{i}\pi_{ki}-\sum_{j=1,2}u^{j}\Gamma_{3j}^{k} \notag
\\
&=&-\frac{1}{\zeta}u^{k}+\sum_{i=1,2}\pi_{ik}u^{i}-\sum_{j=1,2}u^{j}\Gamma_{3j}^{k}\,.
\label{es-5-03}
\end{eqnarray}
It follows that
\begin{equation*}
\nabla^{\Gamma}(\partial_\nu u^{k})=-\frac{1}{\zeta}
\nabla^{\Gamma}u^{k}
+\nabla^{\Gamma}\big(\sum_{i=1,2}\pi_{ik}u^{i}\big)
-\nabla^{\Gamma}\big(\sum_{j=1,2}u^{j}\Gamma_{3j}^{k}\big),
\end{equation*}
so that
\begin{eqnarray}
&&\sum_{k=1,2}\langle
\nabla^{\Gamma}u^{k},\nabla^{\Gamma}(\partial_\nu u^{k})\rangle\notag\\
&&=-\frac{1}{\zeta}|\nabla^{\Gamma}u|^{2}
+ \sum_{k=1,2}\langle\nabla^{\Gamma}u^{k},
\nabla^{\Gamma}\big(\sum_{i=1,2}\pi_{ik}u^{i}\big)
-\sum_{k=1,2}\nabla^{\Gamma}\big(\sum_{j=1,2}u^{j}\Gamma_{3j}^{k}\big)\rangle
\notag \\
&&\leq
-\frac{1}{\zeta}|\nabla^{\Gamma}u|^{2}+C\left(|\nabla^{\Gamma}u|^{2}
 +|u|^{2}\right).
\label{key-es-01}
\end{eqnarray}
Therefore, we have
\begin{equation*}
I\leq -\frac{1}{\zeta}\int_{\Gamma}|\nabla^{\Gamma }u|^{2}
+C\int_{\Gamma}\big(|\nabla u|^{2}+|u|^{2}\big),
\end{equation*}
where $C$ is a constant depending only on $\Omega$.
Combining this inequality with (\ref{5-1-01}) yields
\begin{equation*}
\|\nabla^2u\|_{2}^{2}+\frac{2}{\zeta}\int_{\Gamma}|\nabla^{\Gamma}u|^{2}
\leq \|\Delta u\|_{2}^{2} +C\int_{\Gamma}\big(|\nabla
u|^{2}+|u|^{2}\big)\text{.}
\end{equation*}
Then the conclusion follows from the Sobolev embedding:
\begin{equation}\label{sobolev-embedding}
\int_{\Gamma}|\nabla u|^{2}\leq \varepsilon
\|\nabla^2u\|_{2}^{2}+\frac{C}{\varepsilon}\|\nabla u\|_{2}^{2}
\end{equation}
for some constant $C=C(\Omega)$, independent of $\varepsilon>0$,
since the boundary $\Gamma$ has bounded geometry.
\end{proof}

As a consequence, we have the following elliptic estimate.

\begin{corollary}\label{the01}
There exists a positive constant $C$ depending only on $\Omega$ such
that
\begin{eqnarray}
\|u\|_{H^{2}}^{2}+\frac{1}{\zeta}\|\nabla^{\Gamma}u\|_{L^{2}(\Gamma)}^{2}
\leq C\|(\nabla\times(\nabla\times u),\, \nabla \times
u,\,u)\|_{2}^{2}
\label{5-1-03}
\end{eqnarray}
for any vector field $u\in H^{2}(\Omega)$ satisfying
\eqref{kinematic-1}--\eqref{nav-f}.
\end{corollary}

Therefore, for any \emph{divergence-free} vector field $u$ on
$\Omega$ satisfying \eqref{kinematic-1}--\eqref{nav-f},
\begin{eqnarray}
C\|(\Delta u,\,\nabla\times u,\, u)\|_{2}^{2} \leq \|u\|_{H^{2}}^{2}
\leq C^{-1} \|(\Delta u,\, \nabla\times u,\, u)\|_{2}^{2}
\label{24-00-1}
\end{eqnarray}
for some constant $C>0$ depending only on the domain $\Omega$, but
independent of $\zeta>0$.

\section{The Stokes operator with the Navier Boundary Condition}

In this section, we develop a theory of the Stokes operator acting
on divergence-free vector fields in a general domain $\Omega$
subject to the kinematic and Navier boundary conditions
\eqref{kinematic-1}--\eqref{nav-f}.

Note that the divergence operator $\nabla\cdot$ defined for smooth
vector fields with compact supports in $\Omega$ is closable in
$L^{2}(\Omega)$.
The kernel, $\ker(\nabla\cdot)$, is a closed subspace of
$L^{2}(\Omega)$, denoted by $K_{2}(\Omega)$. Any vector field $u\in
K_{2}(\Omega)\cap H^{1}(\Omega)$ is divergence-free: $\nabla \cdot
u=0$, and satisfies the kinematic condition \eqref{kinematic-1}.
The orthogonal complement of $K_{2}(\Omega)$ is a closed subspace of
$L^{2}(\Omega)$, denoted by $G_{2}(\Omega)$, and the decomposition
\begin{equation*}
L^{2}(\Omega)=K_{2}(\Omega)\oplus G_{2}(\Omega)
\end{equation*}
is called the Helmholtz decomposition. Any element in
$G_{2}(\Omega)$ can be identified with the gradient of a scalar
function, that is,
$$
G_{2}(\Omega)=\{\nabla p\in L^{2}(\Omega)\,:\, p\in L_{\text{loc}}^{2}(\Omega)\}.
$$
Let
$$
P_{\infty}\, :\, L^{2}(\Omega)\rightarrow K_{2}(\Omega)
$$
be the projection from $L^{2}(\Omega)$ onto $K_{2}(\Omega)$. The
following fact is easy but important.

\begin{lemma}\label{le-0-1}
Let $u\in H^{1}(\Omega)$, and let $u=P_{\infty}(u)+\nabla q_{u}$ be
the Helmholtz decomposition of $u$. Then
$$
\nabla\times P_{\infty}(u)=\nabla \times u, \qquad \nabla\cdot
P_{\infty}(u)=0, \qquad \left.
P_{\infty}(u)^{\bot}\right|_{\Gamma}=0.
$$
\end{lemma}

\begin{proposition}\label{pr-23-1}
Let $u\in H^{1}(\Omega)$. Then
\begin{equation}
\|\nabla P_{\infty }(u)\|_{2}^{2}=\|\nabla\times
u\|_{2}^{2}-\int_{\Gamma}\pi (P_{\infty}(u), P_{\infty}(u)),
\label{23-3-2}
\end{equation}
and
\begin{equation}
\|\nabla P_{\infty}(u)\|_{2}\leq C\|(\nabla\times u, \,u)\|_{2}
\label{23-3-1}
\end{equation}
for some constant $C>0$ depending only on $\Omega$.
\end{proposition}

\begin{proof}
By Lemma \ref{le-0-1}, we have $\nabla\times
P_{\infty}(u)=\nabla\times u$. Following the elliptic estimate
(\ref{07june06-01}),
\begin{eqnarray*}
\|\nabla P_{\infty}(u)\|_{2} \leq C\|(\nabla\times P_{\infty}(u),\,
P_{\infty}(u))\|_{2}\leq C\|(\nabla\times u, u)\|_{2},
\end{eqnarray*}
which gives (\ref{23-3-1}). Then (\ref{23-3-2}) follows from
integration by parts and (\ref{23-3-1}).
\end{proof}

The Stokes operator $S$ can be defined to be the composition
$S=P_{\infty }\circ \Delta $ with domain $H^2(\Omega)$. We often
restrict the Stokes operator on the Hilbert space $K_{2}(\Omega)$,
hence with domain $H^{2}(\Omega)\cap K_{2}(\Omega)$, but we will use
the same notation $S$ if no confusion may arise.

\subsection{The Stokes operator with the Navier boundary condition \eqref{nav-f}}

Let $\zeta >0$ be a constant, and let $D_{0,\zeta}(S)$ be the space
of all vector fields $u\in K_{2}(\Omega)\cap C^{\infty}(\Omega)$ (so
that $\nabla\cdot u=0$ and $u^{\bot}\big|_{\Gamma}=0$) which satisfy
the Navier's $\zeta$-condition \eqref{nav-f}.
Then $D_{0,\zeta}(S)$ is dense in $K_{2}(\Omega)$, and
$(S,D_{0,\zeta}(S))$ is a densely defined linear operator on the
Hilbert space $K_{2}(\Omega)$.

\begin{lemma}\label{le-0-2}
Let $u\in D_{0,\zeta}(S)$, and let $\Delta u=S(u)+\nabla p$ be the
Helmholtz decomposition of $\Delta u$. Then $p$ is the unique
solution (up to a constant) of the Neumann boundary problem:
\begin{equation}
\Delta p=\nabla\cdot\left(\Delta u\right)\text{, } \qquad
\partial_\nu p\big|_{\Gamma}
=\frac{1}{\zeta}\nabla^{\Gamma}\cdot
u-2\nabla^{\Gamma}\cdot\pi(u)\text{.}
\label{p-30-02}
\end{equation}
\end{lemma}

\begin{proof}
Since $\Delta u=Su+\nabla p$, then, by taking the divergence and
considering the normal components, we can easily see that $p$
satisfies the Poisson equation:
\begin{equation}
\Delta p=\nabla\cdot\left(\Delta u\right)\text{,} \qquad
\partial_\nu p\big|_{\Gamma}=\langle\Delta u, \nu \rangle\text{.}
\label{p-30-01}
\end{equation}
Since $\nabla\cdot u=0$, $\Delta u=-\nabla\times(\nabla\times u)$ so
that
\begin{eqnarray*}
\langle \Delta u, \nu\rangle=-\langle\nabla\times(\nabla\times u),
\nu\rangle =-\nabla^{\Gamma}\times (\nabla\times u)^{\Vert},
\end{eqnarray*}
and (\ref{p-30-02}) follows from the Navier's $\zeta$-condition
\eqref{nav-f}.
\end{proof}

\begin{remark}
For any $p\geq 2$, there exists a constant $C(p)>0$ depending only
on $\Omega $ (e.g., $C(2)=1$) such that
$$
\|\nabla p\|_{p}\leq C\|\Delta u\|_{p}.
$$
Therefore,
$$
\|S(u)\|_{p}\leq C(p)\|\Delta u\|_{p} \qquad\,\, \text{for any $u\in
D_{0,\zeta}(S)$}.
$$
Of course, $\|S(u)\|_{2}=\|P_{\infty}\Delta u\|_{2} \leq \|\Delta
u\|_{2}$, if $u\in H^{2}(\Omega)$.
\end{remark}

\begin{theorem}\label{th-n-01}
Consider the bilinear form $(\mathcal{E},D_{0,\zeta }(S))$ on
$K_{2}(\Omega)$:
\begin{equation}
\mathcal{E}(u,w)=-\int_{\Omega}\langle Su,w\rangle \qquad\text{for
any }\,  u,w\in D_{0,\zeta}(S)\text{.}  \label{b-0-01}
\end{equation}
Then
\begin{enumerate}\renewcommand{\theenumi}{\roman{enumi}}

\item The bilinear form $(\mathcal{E},D_{0,\zeta}(S))$ on the Hilbert space
$K_{2}(\Omega)$ is densely definite, symmetric, and
\begin{equation}
\mathcal{E}(u,w)=\int_{\Omega}\langle \nabla u,\nabla w\rangle
+\frac{1}{\zeta}\int_{\Gamma}\langle u,w\rangle -\int_{\Gamma}\pi
(u,w) \qquad\text{for any}\,\, u,w\in D_{0,\zeta}(S).
\label{bi-0-21}
\end{equation}

\item For any $\varepsilon \in (0,1)$, there exists a constant
$C(\varepsilon,\Omega)$ such that
\begin{equation}
\mathcal{E}(u,u)\geq (1-\varepsilon )\|\nabla
u\|_{2}^{2}-C(\varepsilon, \Omega)\|u\|_{2}^{2}\,\, \qquad\text{for
any}\,\, u\in D_{0,\zeta}(S). \label{0.0.2}
\end{equation}

\item $(\mathcal{E},D_{0,\zeta }(S))$ is closable on $K_{2}(\Omega)$, its
closure is denoted by $(\mathcal{E},D_{\zeta}(\mathcal{E}))$.
Identity {\rm (\ref{bi-0-21})} remains true for any $u,w\in
D_{\zeta}(\mathcal{E})$.

\item If $\pi \leq \frac{1}{\zeta }$, then
\begin{equation}
\mathcal{E}(u,u)\geq \|\nabla u\|_{2}^{2}\qquad \text{for any}\,\,
u\in D_{\zeta }(\mathcal{E})\text{.} \label{0-0-1}
\end{equation}

\item $D_{\zeta}(\mathcal{E})=K_{2}(\Omega)\cap H^{1}(\Omega)$
which is thus independent of $\zeta$ and hence denoted by
$D(\mathcal{E})$.
\end{enumerate}
\end{theorem}

\begin{proof}
Let $u,w\in D_{0,\zeta}(S)$ and write $\Delta u=S(u)+\nabla p$.
Since $\nabla\cdot u=0$, then
\begin{equation}\label{s-eq}
S(u)=-\nabla \times (\nabla \times u)-\nabla p,
\end{equation}
where $p$ solves the Neumann problem (\ref{p-30-02}). Taking inner
product on both sides of \eqref{s-eq} with $w$ and integration by
parts on $\Omega$ yields
\begin{eqnarray*}
\mathcal{E}(u,w) &=&\int_{\Omega}\langle \nabla \times (\nabla\times
u),w\rangle +\int_{\Omega}\langle \nabla p,w\rangle  \\
&=&\int_{\Omega}\langle \nabla \times u,\nabla \times w\rangle
+\int_{\Gamma}\langle (\nabla \times u)^{\Vert }\times w,\nu \rangle  \\
&=&\int_{\Omega}\langle \nabla \times u,\nabla \times w\rangle
-2\int_{\Gamma}\pi (u,w)+\frac{1}{\zeta }\int_{\Gamma}\langle
u,w\rangle,
\end{eqnarray*}
where we have used
\eqref{kinematic-1}--\eqref{nav-f} so that
\begin{equation}
\langle (\nabla\times u)^{\Vert }\times w,\nu \rangle
=\frac{1}{\zeta} \langle u,w\rangle -2\pi (u,w)\text{.}
\label{4-30-03}
\end{equation}
Therefore, $(u,w)\rightarrow \mathcal{E}(u,w)$ is symmetric and
bilinear. Since $\left. u^{\bot}\right|_{\Gamma}=\left.
w^{\bot}\right|_{\Gamma}=0$, then
\begin{equation*}
\int_{\Omega}\langle \nabla u,\nabla w\rangle =\int_{\Omega}\langle
\nabla \times u,\nabla \times w\rangle -\int_{\Gamma}\pi \left(u,
w\right),
\end{equation*}
and hence
\begin{equation}
\mathcal{E}(u,w)=\int_{\Omega }\langle \nabla u,\nabla w\rangle
+\frac{1}{\zeta}\int_{\Gamma }\langle u,w\rangle -\int_{\Gamma }\pi
(u,w)\text{.} \label{4-30-07}
\end{equation}
If $\pi\leq \frac{1}{\zeta}$, then
\begin{equation*}
\mathcal{E}(u,u)\geq \|\nabla u\|_{2}^{2}\text{ .}
\end{equation*}

Let $\lambda_{1}$ be a upper bound of the second fundamental form
$\pi$, i.e., $\pi\leq \lambda_{1}$, then
\begin{eqnarray*}
\mathcal{E}(u,u)
\geq \|\nabla u\|_{2}^{2}-\lambda_{1} \int_{\Gamma}|u|^{2}
\geq\left( 1-\varepsilon \right)\|\nabla u\|_{2}^{2}
  -\frac{C}{\varepsilon}\|u\|_{2}^{2}\text{}
\end{eqnarray*}
for some $C=C(\Omega)>0$, where we have used the trace imbedding
inequality:
\begin{equation}\label{sobolev-trace-embedding}
\int_\Gamma |u|^2\le \varepsilon\|\nabla u\|^2_2+
\frac{C}{\varepsilon}\|u\|_2^2.
\end{equation}

Next, we show that $(\mathcal{E},D_{0,\zeta}(S))$ is closable on
$K_{2}(\Omega)$. Indeed, if $u_{n}\in D_{0,\zeta}(S)$ such that
$\mathcal{E}(u_{n}-u_{m},u_{n}-u_{m})\rightarrow 0$ and
$\|u_{n}-u_{m}\|_{2}\rightarrow 0$, then, since
\begin{eqnarray*}
\frac{1}{2}\|\nabla (u_{n}-u_{m})\|_{2}^{2}\leq
C\|u_{n}-u_{m}\|_{2}^{2} +\mathcal{E}(u_{n}-u_{m},u_{n}-u_{m}),
\end{eqnarray*}
we have
$$
\|\nabla (u_{n}-u_{m})\|_{2}^{2}\rightarrow 0.
$$
That is, $\{u_{n}\}$ is a Cauchy sequence in $H^{1}(\Omega)$, and
hence there exists a unique $u\in H^{1}(\Omega)$ such that
\begin{equation*}
\|u_{n}-u\|_{2}^{2}+\|\nabla (u_{n}-u)\|_{2}^{2}\rightarrow
0\text{.}
\end{equation*}
It follows by the Sobolev imbedding that
\begin{equation*}
\lim_{n\rightarrow \infty}\int_{\Gamma }\langle u_{n},u_{n}\rangle
=\int_{\Gamma }|u|^{2}\text{, } \qquad
\lim_{n\rightarrow\infty}\int_{\Gamma}\pi(u_{n},u_{n})=\int_{\Gamma}\pi
(u,u)
\end{equation*}
so that
\begin{equation*}
\lim_{n\rightarrow \infty }\mathcal{E}(u_{n},u_{n})=\int_{\Omega
}|\nabla u|^{2}+\frac{1}{\zeta }\int_{\Gamma }|u|^{2}-\int_{\Gamma
}\pi (u,u),
\end{equation*}
and $u$ belongs to the closure of $(\mathcal{E},D_{0,\zeta }(S))$.

Finally, we prove that $D_{\zeta}(\mathcal{E})=K_{2}(\Omega)\cap
H^{1}(\Omega)$, which is not surprising, since the Navier's
$\zeta$-condition \eqref{nav-f}
that has to be satisfied for any $u\in D_{0,\zeta}(S)$ will be
``forgot'' when passing to the limit in $H^{1}(\Omega)$ (in which
the boundary values of the first derivative can not be retained).
Therefore, $D_{0,\zeta}(S)$ is dense in $K_{2}(\Omega)\cap
H^{1}(\Omega)$ in the $H^{1}$-norm.

To see this, consider the case that
$\Omega=\{(x_{1},x_{2},x_{3}):x_{3}>0\}$, and
$u=(u^{1},u^{2},u^{3})\in D_{0,\zeta}(S)\cap
C_{0}^{2}(\mathbb{R}^{3})$ such that $\nabla\cdot u=0$ and
$u^{3}(x_{1},x_{2},0)=0$. Then, for every $\varepsilon>0$, choose
\begin{eqnarray*}
(u_{\varepsilon}^{1}, u_\varepsilon^2, u_\varepsilon^3)
=(u^1,u^2,u^3)+x_3\mathcal{X}_{\{|x_{3}|<\varepsilon\}}(h_1,h_2,h_3).
\end{eqnarray*}
Then
\begin{eqnarray*}
\left. (\nabla\times u_{\varepsilon})^{1}\right|_{\Gamma} &=&\left.
\partial_{x_{3}}u^{2}\right|_{\Gamma}
=\partial_{x_3}u^{2}(x_{1},x_{2},0)+h_{2}(x_{1},x_{2})\text{,} \\
\left. (\nabla \times u_{\varepsilon})^{2}\right|_{\Gamma}
&=&-\left. \partial_{x_{3}}u^{1}\right|_{\Gamma}
=-\partial_{x_3}u^{1}(x_{1},x_{2},0)-h_{1}(x_{1},x_{2})\text{,} \\
\nabla\cdot u_{\varepsilon}
&=&\mathcal{X}_{\{|x_{3}|<\varepsilon\}}\left(x_{3}\nabla\cdot
h+h_{3}\right) \text{.}
\end{eqnarray*}
To match the Navier's $\zeta$-condition, we set
\begin{eqnarray*}
h_{j}(x_{1},x_{2})
&=&\frac{1}{\zeta}u^{j}(x_{1},x_{2},0)
-\partial_{x^3}u^{j}(x_{1},x_{2},0),\qquad j=1,2, \\
h_{3}(x_{1},x_{2}) &=&-x_{3}\left(\partial_{x_1} h_{1} +
\partial_{x_2} h_{2}\right) \text{.}
\end{eqnarray*}
Hence, $u_{\varepsilon}\in D_{\zeta}(\mathcal{E})$ and
$u_{\varepsilon}\rightarrow u$ in $H^{1}$, which concludes the
proof.
\end{proof}

\begin{corollary}\label{co-n-01}
$\left(\mathcal{E},D(\mathcal{E})\right)$ is a densely defined,
bounded below, and closed symmetric form on the Hilbert space
$K_{2}(\Omega)$. Moreover,
\begin{equation*}
\mathcal{E}(u,u)=\|\nabla u\|_{2}^{2}+\frac{1}{\zeta
}\|u\|_{L^{2}(\Gamma )}^{2}-\int_{\Gamma}\pi (u,u)\qquad\text{for
any $u\in D_\zeta(\mathcal{E})$},
\end{equation*}
and there exists $M(\varepsilon,\zeta)>0$ such that
\begin{equation}
\|\nabla u\|_{2}^{2}\leq (\mathcal{E}+\Lambda I)(u,u)\leq
(1+\varepsilon )\|\nabla u\|_{2}^{2}+
M(\varepsilon,\zeta)\|u\|_{2}^{2}\ \qquad\text{for all $u\in
D_\zeta(\mathcal{E})$}. \label{00-01}
\end{equation}
\end{corollary}

\begin{proof}
Suppose that $\pi \geq -C_{0}$ for some $C_{0}\geq 0$. Then
\begin{equation*}
\mathcal{E}(u,u)\leq \|\nabla
u\|_{2}^{2}+\big(\frac{1}{\zeta}+C_{0}\big)\|u\|_{L^{2}(\Gamma)}^{2}\text{.}
\end{equation*}
The Sobolev imbedding yields that, for every $\varepsilon\in (0,1)$,
there exists  $C_{0}>0$ such that
\begin{equation*}
\big(\frac{1}{\zeta }+C_{0}\big)\|u\|_{L^{2}(\Gamma )}^{2}\leq
\varepsilon \|\nabla u\|_{2}^{2}+M(\varepsilon,\zeta)\|u\|_{2}^{2}
\end{equation*}
so that (\ref{00-01}) follows.
\end{proof}

\begin{definition}\label{4.7-a}
Let $\zeta>0$. Then the unique self-adjoint operator on
$K_{2}(\Omega)$ associated with the closed symmetric form
$(\mathcal{E}, D_\zeta(\mathcal{E}))$ is denoted again by $S$, with
its domain $D_{\zeta}(S)$, called the Stokes operator with Navier's
$\zeta$-condition, or simply the Stokes operator if no confusion may
arise.
\end{definition}

According to Definition \ref{4.7-a}, $(S, D_{\zeta}(S))$ is the
unique self-adjoint operator on $K_{2}(\Omega)$ such that
\begin{equation*}
\mathcal{E}(u,w)=-\int_{\Omega}\langle Su,w\rangle\qquad\text{for
any}\,\,\, u\in D_{\zeta}(S), \,\,  w\in D_\zeta(S)\text{,}
\end{equation*}
and
\begin{equation*}
D_{0,\zeta}(S)\subset D_{\zeta}(S)\subset H^{1}(\Omega)\cap
K_{2}(\Omega)\text{.}
\end{equation*}
Moreover, if $u\in D_{\zeta}(S)$, then $u\in H^{1}(\Omega)$ with
$\nabla\cdot u=0$ and $\left. u^{\bot}\right|_{\Gamma}=0$. In
particular, there exists $\Lambda\ge 0$ such that
 $-S+\Lambda I$ is positive definite (when
$\pi\leq\frac{1}{\zeta}$, $\Lambda=0$).

To end this section, we establish an $L^{2}$-estimate for the total
derivative of $S(u)$.

\begin{lemma}\label{23-01-1}
Let $p_{u}$ be the unique solution (up to a constant) of the Neumann
problem:
\begin{equation}
\Delta p_{u}=0\text{,}\qquad \partial_\nu
p_{u}\big|_{\Gamma}=\frac{1}{\zeta}\nabla^{\Gamma}\cdot u
-2\nabla^{\Gamma}\cdot \pi(u) \label{23-4-2}
\end{equation}
for $u\in H^{2}(\Omega)$ satisfying $\left.
u^{\bot}\right|_{\Gamma}=0$. Then, for every $\varepsilon >0$, there
exists a constant $M(\varepsilon, \zeta)$ depending only on
$\varepsilon$, $\zeta$, and the domain $\Omega $ such that
\begin{equation}
\|\nabla p_{u}\|_{2}\leq \varepsilon \|\nabla\times\nabla\times
u\|_{2}+ M(\varepsilon, \zeta)\|u\|_{2}\text{.}  \label{23-02-2}
\end{equation}
\end{lemma}

\begin{proof}
By integration by parts, one obtains
\begin{eqnarray*}
\|\nabla p_{u}\|_{2}^{2}&=&-\int_{\Omega}p_{u}\Delta
p_{u}+\int_{\Gamma}p_{u}\partial_\nu p_{u} \\
&=&\frac{1}{\zeta}\int_{\Gamma}p_{u}\nabla^{\Gamma}\cdot
u-2\int_{\Gamma}p_{u}\nabla^{\Gamma}\cdot \pi(u) \\
&\leq & \|p_{u}\|_{L^{2}(\Gamma)}M\big(\varepsilon\|\nabla^2u\|_{2}
  +\|u\|_{H^{1}}\big) \\
&\leq &\|\nabla p_{u}\|_{L^{2}(\Gamma)}\big(\varepsilon
\|\nabla^2u\|_{2}+\|u\|_{H^{1}}\big)
\end{eqnarray*}
for some constant $M>0$ which may depend on $\zeta$ and
$\varepsilon$. Then (\ref{23-02-2}) follows from
(\ref{07june06-01}).
\end{proof}

\begin{proposition}\label{prop-23-41}
Let $u\in D_{0,\zeta}(S)$, $\omega =\nabla\times u$, and
$\psi=\nabla\times\omega =-\Delta u$. Then
\begin{equation}
\|S(u)\|_{H^{1}}\leq M\big(\|\nabla \times \psi\|_{2}
+\|(\psi,u)\|_{2}\big), \label{23-02-1}
\end{equation}
where $M(\varepsilon, \zeta)>0$ depend only on $\varepsilon$,
$\zeta$, and $\Omega$.
\end{proposition}

\begin{proof}
Since $\nabla\cdot S(u)=0$ and $\left.
S(u)^{\bot}\right|_{\Gamma}=0$, according to (\ref{07june06-01}),
\begin{equation*}
\|S(u)\|_{H^{1}}^{2}\leq C\|(S(u), \nabla\times S(u))\|_{2}^{2},
\end{equation*}
where $C>0$ depends only on $\Omega$.

Let $S(u)=-\psi -\nabla p$, where $p$ solves (\ref{23-4-2}). Then
$\nabla\times S(u)=-\nabla \times \psi$ so that
\begin{equation*}
\|S(u)\|_{H^{1}}^{2}=C\|(\nabla\times\psi, S(u))\|_{2}^{2}.
\end{equation*}
Together with (\ref{23-02-2}), (\ref{23-02-1}) follows immediately.
\end{proof}

\subsection{Spectral theory of the Stokes operator subject to the kinematic and Navier
boundary conditions \eqref{kinematic-1}--\eqref{nav-f}}
To establish other important properties of the Stokes operator
$(S,D_{\zeta }(S))$, we study the boundary problem of the Stokes
equation:
\begin{equation}
\lambda u+\Delta u-\nabla p=f\text{,}\qquad \nabla\cdot u=0
\label{1.0.8}
\end{equation}
subject to the boundary conditions
\eqref{kinematic-1}--\eqref{nav-f},
where $f\in K_{2}(\Omega )\cap C^{\infty }(\Omega)$ and $\lambda \in
\mathbb{R}$ is a constant.

Taking the divergence to both sides of equation (\ref{1.0.8}) yields
that the scalar function $p$ satisfies the Neumann problem of the
Laplace equation:
\begin{equation}
\Delta p=0\text{,} \qquad  \partial_\nu p\big|_{\Gamma}
=\frac{1}{\zeta}\nabla^{\Gamma}\cdot u-2\nabla^{\Gamma}\cdot\pi(u)
\text{.}  \label{1.0.10}
\end{equation}

Let $\Lambda>0$ be the constant such that $-S+\Lambda I\ge 0$. Then,
for $\lambda >\Lambda$, let $R_{\lambda}$ denote the resolvent,
i.e., $R_{\lambda}=\left(\lambda I-S\right)^{-1}$, which is a
bounded linear operator on $K_{2}(\Omega)$.

\begin{theorem}\label{th-n-07}
For any $\lambda >\Lambda$, $R_{\lambda}$ is a compact operator on
$K_{2}(\Omega)$.
\end{theorem}

\begin{proof}
Let $f\in K_{2}(\Omega )$ and $u=R_{\lambda }f$. Then $u\in D_{\zeta
}(S)$ and
\begin{equation*}
\left( \lambda I-S\right) u=f\text{. }
\end{equation*}
Suppose in addition that $u\in D_{0,\zeta}(S)$, so that $\nabla\cdot
u=0$, and $u$ satisfies \eqref{kinematic-1}--\eqref{nav-f}. Let
$\Delta u=Su+\nabla p$. Then
\begin{equation}
\lambda u-\Delta u+\nabla p=f\text{, }  \label{0.0.-4}
\end{equation}
and $p$ solves the Neumann problem (\ref{1.0.10}). Hence, for every
$\varepsilon >0$,
\begin{equation*}
\|\nabla p\|_{2}\leq \varepsilon \|\nabla^2u\|_{2} +M(\varepsilon,
\zeta)\|u\|_{H^{1}}
\end{equation*}
for some constant $M(\varepsilon, \zeta)$. It is easy to devise the
energy estimate for $u$. Indeed, since
\begin{equation*}
\lambda \langle u,  u\rangle-\langle u,Su\rangle =\langle
u,f\rangle,
\end{equation*}
then
\begin{equation*}
\lambda \|u\|_{2}^{2}-\int_{\Omega}\langle u, Su\rangle
=\int_{\Omega}\langle u, f\rangle \leq \|u\|_{2}\|f\|_{2}\text{.}
\end{equation*}
Furthermore, since $-S+\Lambda\ge 0$, i.e., $-\int_{\Omega}\langle
u, Su\rangle \geq \|\nabla u\|_{2}^{2}-\Lambda \|u\|_{L^2}$, we have
\begin{equation*}
(\lambda-\Lambda)\|u\|_{2}^{2}+\|\nabla u\|_{2}^{2}\leq
\|u\|_{2}\|f\|_{2}.
\end{equation*}
Therefore, we have
\begin{equation}
\frac{3(\lambda-\Lambda)}{4}\|u\|_{2}^{2}+\|\nabla u\|_{2}^{2} \leq
\frac{1}{\lambda-\Lambda}\|f\|_{2}^{2}\text{ .}  \label{e-1.-1}
\end{equation}
Next, we  estimate the second-order derivative of $u$. Since $u$
satisfies \eqref{kinematic-1}--\eqref{nav-f}, according to the
elliptic estimate (Theorem 3.1):
\begin{equation}
\|\nabla^2
u\|_{2}^{2}+\frac{2}{\zeta}\|\nabla^{\Gamma}u\|_{L^{2}(\Gamma)}^{2}
\leq C\left( \|\Delta u\|_{2}^{2}+\|u\|_{H^{1}}^{2}\right) \text{,}
\label{0.0.3-1}
\end{equation}
then
\begin{eqnarray*}
\|\nabla^2
u\|_{2}^{2}+\frac{2}{\zeta}\|\nabla^{\Gamma}u\|_{L^{2}(\Gamma)}^{2}
&\leq &C\left(\|\Delta u\|_{2}^{2}+\|u\|_{H^{1}}^{2}\right)
\\
&\leq &C\left(\|\nabla
p\|_{2}^{2}+\|f\|_{2}^{2}+\|u\|_{H^{1}}^{2}\right)
\\
&\leq &\varepsilon \|\nabla^2 u\|_{2}^{2}+M(\varepsilon,
\zeta)\left(\|u\|_{H^{1}}^{2}+\|f\|_{2}^{2}\right) \text{.}
\end{eqnarray*}
It follows that
\begin{eqnarray*}
\|\nabla^2 u\|_{2}^{2}\leq M(\varepsilon, \zeta, \lambda)\left(
\|u\|_{H^{1}}^{2}+\|f\|_{2}^{2}\right) \leq M(\varepsilon, \zeta,
\lambda,\Lambda)\|f\|_{2}^{2}\text{,}
\end{eqnarray*}
so that
\begin{equation}
\|u\|_{H^{2}}\leq M(\zeta, \lambda, \Lambda)\|f\|_{2}^{2}
\label{0.0.6}
\end{equation}
for some constant $M>0$ depending only on $\zeta, \lambda$, and
$\Lambda$. Therefore, $R_{\lambda}$ is compact.
\end{proof}

\begin{theorem}\label{th-n-21}
The spectrum of the Stokes operator $(S, D_{\zeta}(S))$ with
Navier's $\zeta$-condition \eqref{nav-f} is discrete and belongs to
$(-\infty, \Lambda]$ for some constant
$\Lambda=\Lambda(\Omega,\zeta)$. The eigenvalues
 $\lambda_{j}\le \Lambda$ can be ordered as
\begin{equation*}
\Lambda\geq \lambda_{0}\geq \lambda_{1}\geq \cdots \geq
\lambda_{n}\geq \cdots \text{, \ } \quad \lambda_{n}\downarrow
-\infty.
\end{equation*}
Moreover, there are eigenfunctions $a_n$: $Sa_{n}=\lambda_{n}a_{n}$,
where $a_{n}\in D_{0,\zeta }(S)$ (so that $a_{n}$ satisfy
\eqref{kinematic-1}--\eqref{nav-f}) such that $\{a_{n}\,:\,n\geq
0\}$ is a complete orthonormal basis of $K_{2}(\Omega)$. In
particular, when $\pi\le \frac{1}{\zeta}$, $\Lambda=0$, i.e.,
$\lambda_j\leq 0$ for $j=0,1,2,\cdots$.
\end{theorem}

\begin{proof}
The standard spectral theory for self-adjoint operators yields that
the spectrum of $(S,D_{\zeta}(S))$ belongs to $(-\infty, \Lambda]$
and is discrete, and there exists an orthonormal basis $\{a_{n}\,:\,
n\geq 0\}$, where each $a_{n}\in D_{\zeta}(S)$ and
$Sa_{n}=\lambda_{n}a_{n}$ for the corresponding eigenvalues
$\lambda_{n}$. The standard elliptic theory then implies that
$a_{n}\in D_{0,\zeta}(S)$.
\end{proof}

\subsection{Several facts about the
Navier's $\protect\zeta$-condition \eqref{nav-f}}

In what follows, we will assume that
$\{a_{n}\,:\, n\geq 0\}$ is the
orthonormal basis of $K_{2}(\Omega)$ constructed in Theorem
\ref{th-n-21}.

Let $N$ is an integer, which may be infinity.  Let $X_{N}$ be the
Hilbert space spanned by $\{a_{k}\,:\,k\leq N\}$, and let $P_{N}:
L^{2}(\Omega)\rightarrow X_{N}$ be the projection. That is, for
every $u\in L^{2}(\Omega)$,
\begin{equation}
P_{N}u=\sum_{k=0}^{N}a_{k}\int_{\Omega }\langle a_{k},u\rangle
\text{. }
\label{20-01}
\end{equation}
Of course, $P_{\infty}u=\sum_{k=0}^{\infty}a_{k}\int_{\Omega}\langle
a_{k},u\rangle$ is the projection from $L^{2}(\Omega)$ onto
$K_2(\Omega)$. If $N<\infty$, $P_{N}(u)\in D_{0,\zeta}(S)$ for any
$u\in L^{2}(\Omega)$.

\begin{proposition}\label{le-s-0}
Let $u\in \cup_{N\in \mathbb{N}}X_N$ the vector space spanned by
elements in $X_N$ where $N$ runs over all natural numbers, $\omega
=\nabla \times u$, and $\psi =\nabla\times\omega =-\Delta u$. Then
\begin{equation}
\left. \langle \psi, \nu \rangle \right|_{\Gamma}
=-\frac{1}{\zeta}\nabla^{\Gamma}\cdot u+2\nabla^{\Gamma}\cdot\pi(u),
\label{ss-0-2}
\end{equation}
and
\begin{equation}
\left. \left(\nabla\times\psi\right)^{\Vert}\right|_{\Gamma}
=\frac{1}{\zeta}\big(\ast S(u)\big)-2\big(\ast\pi(S(u))\big)\text{.}
\label{ss-0-7}
\end{equation}
\end{proposition}

\begin{proof}
Let $u\in X_N$ for some $N$.
By definition, $S(u)=\Delta u-\nabla p$ so that $\psi =-S(u)-\nabla
p$. Since $\nabla\cdot S(u)=0$ and $\left.
S(u)^{\bot}\right|_{\Gamma}=0$, then $p$ is the solution of the
Neumann problem:
\begin{equation}
\Delta p=0\text{,}\qquad
\partial_\nu p\big|_{\Gamma}=\frac{1}{\zeta}\nabla^{\Gamma}\cdot u
-2\nabla^{\Gamma}\cdot\pi(u),
\label{ss-0-3}
\end{equation}
and $\left. \langle \psi, \nu\rangle\right|_{\Gamma}=-\left.
\partial_\nu p\right|_{\Gamma}$,  which yields (\ref{ss-0-2}). To
see the tangent component of $\nabla \times \psi $, we note that
$\nabla \times \psi =-\nabla \times S(u)$. Since
$u=\sum_{k=0}^{N}a_{k}\int_{\Omega}\langle a_{k}, u\rangle$,
then
\begin{equation*}
S(u)=\sum_{k=0}^{N}S(a_{k})\int_{\Omega}\langle a_{k},u\rangle
=-\sum_{k=0}^{N}\lambda_{k}a_{k}\int_{\Omega }\langle
a_{k},u\rangle\, .
\end{equation*}
Therefore,
\begin{equation*}
\nabla\times \psi =-\nabla\times
S(u)=\sum_{k=0}^{N}\lambda_{k}\left(\nabla\times a_{k}\right)
\int_{\Omega}\langle a_{k},u\rangle,
\end{equation*}
and it follows that
\begin{eqnarray*}
\left. \left(\nabla\times\psi \right)^{\Vert}\right|_{\Gamma}
&=&-\sum_{k=0}^{N}\lambda_{k}\left(\nabla \times
a_{k}\right)^{\Vert}\int_{\Omega}\langle a_{k},u\rangle \\
&=&\sum_{k=0}^{N}\lambda_{k}\Big(\frac{1}{\zeta}(\ast
a_{k})-2\big(\ast\pi(a_{k})\big)\Big)\int_{\Omega }\langle a_{k},u\rangle \\
&=&\frac{1}{\zeta}\Big(\ast\big(\sum_{k=0}^{N
}\lambda_{k}a_{k}\int_{\Omega }\langle a_{k},u\rangle\big)\Big)
-2\big(\ast \pi (\sum_{k=0}^{N }\lambda
_{k}a_{k}\int_{\Omega }\langle a_{k},u\rangle)\big) \\
&=&\frac{1}{\zeta }\big(\ast S(u)\big)-2\big(\ast \pi (S(u))\big),
\end{eqnarray*}
which yields the claim.
\end{proof}

\begin{lemma}\label{le-23-4}
For every $\varepsilon>0$, there exists  $M(\varepsilon, \zeta)>0$
such that
\begin{equation}
\frac{1}{2}\int_{\Gamma}\partial_\nu\big(|\psi|^{2}\big)\leq
\varepsilon \|\nabla^3u\|_{2}^{2}+ M\|(\psi, u)\|_{2}^{2}
\label{23-5-03}
\end{equation}
for any $u\in \cup_{N\in \mathbb{N}}X_N$, $\omega =\nabla\times u$,
and $\psi =\nabla\times\omega =-\Delta u$.
\end{lemma}

\begin{proof}
Recall (cf. (\ref{june21-045})) that
\begin{eqnarray*}
\frac{1}{2}\partial_\nu\big(|\psi|^{2}\big)&=&-\nabla^{\Gamma}\cdot\big(
\langle\psi, \nu\rangle \psi^{\Vert}\big) -\pi(\psi^{\Vert}, \psi^{\Vert})
 -H|\psi^{\bot}|^{2} \\
&&+2\langle \psi^{\Vert}, \nabla^{\Gamma}\langle \psi, \nu\rangle
\rangle +\langle\psi^{\Vert}\times
(\nabla\times\psi)^{\Vert},\nu\rangle \text{.}
\end{eqnarray*}
Integrating the equation over $\Omega$ and using the boundary data
in Lemma \ref{le-s-0} yield
\begin{eqnarray*}
\frac{1}{2}\int_{\Gamma}\partial_\nu\big(|\psi|^{2}\big)
&=&-\int_{\Gamma}\nabla^{\Gamma}\cdot\big(\langle\psi, \nu\rangle
\psi^{\Vert}\big)-\int_{\Gamma}\pi(\psi^{\Vert}, \psi^{\Vert})
-\int_{\Gamma}H|\psi^{\bot}|^{2} \\
&&+\frac{1}{\zeta}\int_{\Gamma}\langle \psi^{\Vert}, S(u)\rangle
-2\int_{\Gamma}\pi(\psi^{\Vert}, S(u)) \\
&&-\frac{2}{\zeta}\int_{\Gamma}\langle \psi^{\Vert},
\nabla^{\Gamma}\left(\nabla^{\Gamma}\cdot u\right)\rangle
+4\int_{\Gamma}\langle \psi^{\Vert},
\nabla^{\Gamma}\left(\nabla^{\Gamma}\cdot \pi(u)\right)
\rangle\text{.}
\end{eqnarray*}
The first integral vanishes by Stokes' theorem applying to the
surface $\Gamma$. Using the H\"{o}lder inequality, it follows that
\begin{eqnarray}
\frac{1}{2}\int_{\Gamma}\partial_\nu\big(|\psi|^{2}\big)\leq
M\|\psi^{\Vert}\|_{L^{2}(\Gamma)}^{2}\big(1+
\|S(u)\|_{L^{2}(\Gamma)}+\|(\nabla^2 u, \nabla u,
u)\|_{L^{2}(\Gamma)}\big). \label{4.26-b}
\end{eqnarray}
All the boundary integrals on the right-hand side can be estimated
via the Sobolev imbedding:
\begin{eqnarray*}
\|(\nabla^2 u, \nabla u, u)\|_{L^{2}(\Gamma)}\leq \varepsilon
\|\nabla^3 u\|_{2}+M\|(\psi,u)\|_{2}
\end{eqnarray*}
and
\begin{eqnarray*}
\|S(u)\|_{L^{2}(\Gamma)} &\leq &\varepsilon \|\nabla
S(u)\|_{2}+M\|S(u)\|_{2}
\leq \varepsilon \|\nabla\times\psi\|_{2}+ M\|(\psi, u)\|_{2},
\end{eqnarray*}
where the second inequality follows from (\ref{23-02-1}). Again, by
the Sobolev imbedding,
\begin{equation*}
\|\psi^{\Vert}\|_{L^{2}(\Gamma)}^{2}\leq \varepsilon
\|\nabla\psi\|_{2}^{2}+M\|\psi\|_{2}^{2}\text{.}
\end{equation*}
Plugging these estimates into \eqref{4.26-b},
we conclude
\begin{eqnarray*}
\frac{1}{2}\int_{\Gamma}\partial_\nu\big(|\psi|^{2}\big)\leq \big(
\varepsilon \|\nabla\psi\|_{2}+C_{2}\|\psi\|_{2}\big)\, \big(
\varepsilon \|\nabla^3 u\|_{2}+C\|(\psi, u)\|_{2}+1\big),
\end{eqnarray*}
and (\ref{23-5-03}) follows.
\end{proof}

\begin{theorem}\label{le-23-6}
Let $u\in \cup_{N\in \mathbb{N}}X_N$, $\omega =\nabla\times u$, and
$\psi=-\Delta u$. Then
\begin{equation*}
\|\nabla^3 u\|_{2}\leq M\big(\|\nabla\psi\|_2+\|u\|_{H^{2}}\big)
\end{equation*}
so that
\begin{equation}
\|u\|_{H^{3}}\leq M\|(\nabla\psi, \psi, u)\|_{2}, \label{23-6-01}
\end{equation}
where $M>0$ is a constant depending only on $\zeta$ and the domain
$\Omega$, which may be different in each occurrence.
\end{theorem}

\begin{proof}
First, we need to apply integration by parts to reduce the
$L^{2}$-norm of the third total derivative $\nabla^3 u$ into the
$L^{2}$-norm of the total derivative $\nabla \psi$, plus some
boundary integrals, which can be dominated by the use of boundary
conditions on $u$ and the facts that $\nabla\cdot u=0$ and
$\nabla\cdot\psi =0$. Let us carry out this step in the ordinary
coordinate system, and thus $\partial_{j}=\frac{\partial}{\partial
x^{j}}$ are the coordinate differentiations. Then
\begin{equation*}
|\nabla^3u|^{2}=\big(\partial_{i}\partial_{j}\partial_{k}u^{l}\big)
\big(\partial_{i}\partial_{j}\partial_{k}u^{l}\big),
\end{equation*}
where, as usual, the repeated indices are summed up from $1$ to $3$.
Therefore, by using integration by parts twice,
\begin{eqnarray}
\|\nabla^3 u\|_{2}^{2}
&=&\int_{\Omega}\big(\partial_{k}\psi^{l}\big)\big(\partial_{k}\psi^{l}\big)
+\int_{\Gamma}\big(\partial_{k}\psi^{l}\big)
\big(\partial_{j}\partial_{k}u^{l}\big)\langle\partial_{j}, \nu\rangle
+\int_{\Gamma}\big(\partial_{j}\partial_{k}u^{l}\big)
\big(\partial_{i}\partial_{j}\partial_{k}u^{l}\big) \langle
\partial_{i},\nu \rangle  \notag \\
&=&\int_{\Omega}|\nabla\psi|^{2}
+\int_{\Gamma}\big(\partial_{k}\psi^{l}\big)
\big(\partial_{j}\partial_{k}u^{l}\big)\langle\partial_{j}, \nu\rangle
+\frac{1}{2}\int_{\Gamma}\partial_\nu\big(|\nabla^2u|^{2}\big),
\label{23-01}
\end{eqnarray}
where $\psi^{l}=-\Delta u^{l}$ has been used.
We first handle the last boundary integral
$\frac{1}{2}\int_{\Gamma}\partial_\nu\big(|\nabla^2 u|^{2})$. To
this end, we use a moving frame $(\nabla_{1},\nabla_2,\nabla_{3})$
so that $\nabla_{3}$ coincides with the unit normal $\nu$ and
$\nabla_{i}$, $i=1,2$, are (local) tangent vector fields, so that we
may perform integration by parts on the surface $\Gamma$ for
$\nabla_{i}$, $i=1,2$. Then
\begin{equation*}
J\equiv \frac{1}{2}\partial_\nu\big(|\nabla^2 u|^{2}\big)\,\sim\,
\big(\nabla_{j}\nabla_{k}u^{l}\big)
\big(\nabla_{j}\nabla_{k}\nabla_{3}u^{l}\big),
\end{equation*}
where $\sim$ means that the difference between the two sides
contains only the quadratic terms involving $u$ and its 1st and 2nd
order derivatives, the second fundamental form $\pi$ and its
derivatives. These terms come from the exchange of orders by
applying the derivatives $\nabla_{i}$. According to the Ricci
identity, the difference after the exchange of
$\nabla_{i}\nabla_{j}$ to $\nabla_{j}\nabla_{i}$ is a term which has
order $0$. More precisely,
\begin{equation*}
\nabla_{i}\nabla_{j}T^{\mathbf{l}}-\nabla_{j}\nabla_{i}T^{\mathbf{l}}=\sum_{l=1}^3
C_{ij}^{l}T^{\mathbf{l}},
\end{equation*}
where $T$ is a vector field and $\mathbf{l}$ can be a multi-index,
depending on the order of the vector field $T$. The boundary
integrals of these lower-order terms can be controlled by
\begin{equation*}
C\|(\nabla^2 u, \nabla u, u)\|_{L^{2}(\Gamma)}^{2},
\end{equation*}
which, in turn, can be dominated by
\begin{equation*}
\varepsilon \|\nabla^3 u\|_{2}^{2}+M\|(\psi, u)\|_{2}^{2}\text{.}
\end{equation*}
Note that, in this step, no boundary condition on $u$ is required.
With this principle, without using any boundary condition on $u$, we
may re-group
$\left(\nabla_{j}\nabla_{k}u^{l}\right)
\left(\nabla_{j}\nabla_{k}\nabla_{3}u^{l}\right)$ into the following
items: $\left(\nabla_{j}\nabla_{k}u^{3}\right)
\left(\nabla_{j}\nabla_{k}\nabla_{3}u^{3}\right)$,
$\left(\nabla_{k}\nabla_{3}u^{3}\right)
\left(\nabla_{k}\nabla_{3}^2u^{3}\right)$,
$\left(\nabla_{3}^2u^{3}\right) \left(\nabla_{3}^3u^{3}\right)$,
$\left(\nabla_{j}\nabla_{k}u^{l}\right)
\left(\nabla_{j}\nabla_{k}\nabla_{3}u^{l}\right)$,
$\left(\nabla_{j}\nabla_{3}u^{l}\right)
\left(\nabla_{j}\nabla_{3}^2u^{l}\right)$, and
$\left(\nabla_{3}^2u^{l}\right) \left(\nabla_{3}^3u^{l}\right)$,
where all the other repeated indices run through from $1$ to $2$.
Indeed,
\begin{eqnarray*}
J &\sim &\sum_{j,k=1,2}\big(\nabla_{j}\nabla_{k}u^{3}\big)
\big(\nabla_{j}\nabla_{k}\nabla_{3}u^{3}\big)
+\sum_{l,j,k=1,2}\big(\nabla_{j}\nabla_{k}u^{l}\big) \big(
\nabla_{j}\nabla_{k}\nabla_{3}u^{l}\big) \\
&&+2\sum_{l,j=1,2}\big(\nabla_{j}\nabla_{3}u^{l}\big)
\big(\nabla_{j}\nabla_{3}^2u^{l}\big)
+2\sum_{k=1,2}\big(\nabla_{k}\nabla_{3}u^{3}\big) \big(
\nabla_{k}\nabla_{3}^2u^{3}\big) \\
&&+\big(\nabla_{3}^2u^{3}\big) \big(
\nabla_{3}\nabla_{3}^2u^{3}\big)
+\sum_{l=1,2}\big(\nabla_{3}^2u^{l}\big) \big(
\nabla_{3}^3u^{l}\big).
\end{eqnarray*}
To estimate the boundary integrals of the right-hand side, we use
the kinematic condition \eqref{kinematic-1}, the divergence-free
condition:
\begin{equation}
\nabla_{3}u^{3}=-\sum_{a=1,2}\nabla_{a}u^{a}\text{,}\qquad
\nabla_{3}\psi^{3}=-\sum_{a=1,2}\nabla_{a}\psi^{a}, \label{24-1}
\end{equation}
and the definition of $\psi $:
\begin{eqnarray}
\nabla_{3}^2 u^{b}=\Delta u^{b}-\sum_{a=1,2}\nabla_{a}^2u^{b}
=-\psi^{b}-\sum_{a=1,2}\nabla_{a}^2u^{b}. \label{24-2}
\end{eqnarray}
Both (\ref{24-1}) and (\ref{24-2}) hold on $\Omega$. By using these
relations, we can rewrite $J$ as follows:
\begin{eqnarray*}
J &\sim &-2(\nabla_{j}\nabla_{a}^2u^{l}
+\nabla_{j}\psi^{l})\nabla_{j}\big(\nabla_{3}u^{l}-\nabla_{l}u^{3}\big)
-(\nabla_{a}^2\nabla_{b}u^{b}+\nabla_{a}\psi^{a})[\psi^{3}]
\\
&&+\big(2\nabla_{l}\nabla_{a}u^{a}+\nabla_{a}^2u^{l} +\psi^{l}\big)
\big(\nabla_{l}\psi^{3}\big)
+\big(\nabla_{b}^2u^{l}+\psi^{l}\big)\big(\nabla_{3}\psi^{l}
-\nabla_{l}\psi^{3}\big)\\
&&+\psi^{l}\nabla_{a}^2\big(\nabla_{3}u^{l}-\nabla_{l}u^{3}\big)
+\sum_{l,j,k=1,2}\big(\nabla_{j}\nabla_{k}u^{l}\big) \big(
\nabla_{j}\nabla_{k}(\nabla_{3}u^{l}-\nabla _{l}u^{3}) \big)
\\
&&+\sum_{l,b,a=1,2}\big(\nabla_{b}^2u^{l}\big)
\nabla_{a}^2\big(\nabla_{3}u^{l}-\nabla_{l}u^{3}\big) \text{,}
\end{eqnarray*}
where the other repeated indices are added up through $1$ to $2$.
The terms in the square brackets are to be replaced via the
corresponding boundary conditions, so that the orders of taking
derivatives for these terms are reduced by $1$. Therefore, all the
terms, except the first two, are quadratic forms of $u$ and its 1st
and 2nd order derivatives, and those of $\pi$, so that these terms
are dominated by
\begin{equation*}
C|(\nabla^2 u, \nabla u, u)|^{2}.
\end{equation*}
For the first two terms, we perform integration by parts on
$\Gamma$. The second boundary integral on the right-hand side of
\eqref{23-01} can be estimated similarly.

Let
$I=\frac{1}{2}\int_{\Gamma}\partial_\nu\big(|\nabla^2u|^{2}\big)$.
Integrating the above equation and using integration by parts on the
surface $\Gamma$ yield
\begin{eqnarray*}
I &\leq
&\int_{\Gamma}\big(2\psi^{k}+3\nabla_{k}\nabla_{a}u^{a}+\nabla_{a}^2u^{k}\big)
  \nabla_{k}\psi^{3}+\int_{\Gamma}\big(\psi^{l}+\nabla_{b}^2u^{l}\big)
   \big(\nabla_{3}\psi^{l}-\nabla_{l}\psi^{3}\big) \\
&&+3\int_{\Gamma}\big(\psi^{l}+\nabla_{a}^2u^{l}\big)
  \nabla_{j}^2\big(\nabla_{3}u^{l}-\nabla_{l}u^{3}\big)
+\int_{\Gamma}\big(\nabla_{j}\nabla_{k}u^{l}\big)
  \big(\nabla_{j}\nabla_{k}(\nabla_{3}u^{l}-\nabla_{l}u^{3})\big) \\
&&+\varepsilon \|\nabla^3 u\|_{2}^{2}+M\|(\psi,u)\|_{2}^{2} \text{.}
\end{eqnarray*}
Therefore, using the boundary conditions for $\psi^{3}=\langle\psi,
\nu\rangle$, $\nabla_{3}\psi^{l}-\nabla_{l}\psi^{3}$ (which is
$(\nabla\times\psi)^{\Vert}$), $\nabla_{3}u^{l}-\nabla_{l}u^{3}$
(which is $\left(\nabla\times u\right)^{\Vert}$), together with the
Sobolev imbedding, we have
\begin{equation*}
I\leq \varepsilon\|\nabla^3 u\|_{2}^{2}+M(\varepsilon,\zeta)\|(\psi,
u)\|_{2}^{2}\text{.}
\end{equation*}
\end{proof}

\begin{corollary}
\label{co-23-31} There exists $M(\zeta)>0$ depending only on $\zeta$
and $\Omega$ such that
\begin{eqnarray}
M\|(\nabla\psi, \psi, u)\|_{2}\leq \|u\|_{H^{3}} \leq
M^{-1}\|(\nabla \psi, \psi, u)\|_{2} \label{23-5-6}
\end{eqnarray}
for any $u\in \cup_{N\in \mathbb{N}}X_N$, where $\psi =-\Delta u$.
\end{corollary}

\subsection{The Stokes semigroup}

The self-adjoint operator $S$ on $K_{2}(\Omega)$ has a spectral
decomposition
\begin{equation*}
S=\int_{-\infty}^\Lambda\lambda\,\text{d}E_{\lambda},
\end{equation*}
where $\{E_{\lambda}\,:\,\lambda<\Lambda\}$ is the left-continuous
family of the projection operator $E_{\lambda}$ on the space spanned
by $\{a_{k}\,:\, \lambda_{k}>\lambda\}$. According to the spectral
theory of self-adjoint operators, $u\in K_{2}(\Omega)$ belongs to
$D_{\zeta}(S)$ if and only if
\begin{equation*}
\int_{-\infty}^\Lambda\lambda^{2}\,\text{d}\langle E_{\lambda}u,
u\rangle
=\sum_{k=0}^{\infty}\lambda_{k}^{2}\Big(\int_{\Omega}\langle a_{k},
u\rangle \Big)^{2}<\infty,
\end{equation*}
and $D(\mathcal{E}+\Lambda I)=D\left(\sqrt{-S+\Lambda
I}\right)=K_{2}(\Omega)\cap H^{1}(\Omega)$; and $u\in K_{2}(\Omega)$
is in $D(\mathcal{S})$ if and only if
\begin{equation*}
\int_{-\infty}^\Lambda\lambda\, \text{d}\langle E_{\lambda}u,
u\rangle =\sum_{k=0}^{\infty}\lambda_{k}\Big(\int_{\Omega}\langle
a_{k},u\rangle \Big)^{2}<\infty.
\end{equation*}
In this case,
\begin{equation}
\mathcal{E}(u,u)=-\sum_{k=0}^{\infty}\lambda_{k}\Big(\int_{\Omega}\langle
a_{k},u\rangle \Big)^{2}\text{.} \label{bi-0-41}
\end{equation}

We are going to show the following estimate that plays an important
role in the proof of the existence of strong solutions to the
Navier-Stokes equations.

\begin{theorem}\label{th-n4}
For any $\varepsilon>0$,  there exists $M(\varepsilon, \zeta)$ such
that
\begin{equation}
\|\nabla \times P_{N}(u)\|_{2}^{2}\leq
(1+\varepsilon)\mathcal{E}(P_{\infty}(u),
P_{\infty}(u))+M\|u\|_{2}^{2} \label{20-02}
\end{equation}
for any $u\in L^{2}(\Omega)$ and any integer $N$.
\end{theorem}

\begin{proof}
Recall that
\begin{equation*}
P_{N}(u)=\sum_{k=0}^{N}a_{k}\int_{\Omega}\langle a_{k},u\rangle
=\sum_{k=0}^{N}a_{k}\int_{\Omega}\langle a_{k},P_{\infty}(u)\rangle
\end{equation*}
so that $P_{N}(u)\in D_{0,\zeta}(S)$. According to
(\ref{est-june26-01a}),
\begin{eqnarray*}
\|\nabla \times P_{N}(u)\|_{2}^{2} &=&\|\nabla
P_{N}(u)\|_{2}^{2}+\int_{\Gamma }\pi (P_{N}(u),P_{N}(u)) \\
&\leq &(1+\varepsilon)\|\nabla P_{N}(u)\|_{2}^{2}
+M(\varepsilon,\zeta) \|P_{N}(u)\|_{2}^{2} \\
&\leq & (1+\varepsilon)\mathcal{E}(P_{N}(u),P_{N}(u))+M\|u\|_{2}^{2} \\
&=&-(1+\varepsilon) \int_{\Omega} \langle S\left(
P_{N}(u)\right),P_{N}(u)\rangle +M\|u\|_{2}^{2},
\end{eqnarray*}
where the second inequality follows from (\ref{00-01}). However,
\begin{equation*}
S\left(P_{N}(u)\right)
=\sum_{k=0}^{N}\lambda_{k}a_{k}\int_{\Omega}\langle a_{k},u\rangle.
\end{equation*}
Denote integer $N_0>0$ such that $\Lambda\ge \lambda_0\ge\cdots\ge
\lambda_{N_0}> 0\ge \lambda_{N_0+1}\ge \cdots$. Then we find that,
when $N\ge N_0$,
\begin{eqnarray*}
-\int_{\Omega}\langle S\left(P_{N}(u)\right), P_{N}(u)\rangle
&=&-\sum_{k=0}^{N}\lambda_{k}\int_\Omega\langle a_{k},
P_{N}(u)\rangle
\int_{\Omega}\langle a_{k}, u\rangle  \\
&=&-\sum_{k=0}^{N}\lambda_{k}\big(\int_{\Omega}\langle
a_{k},P_{\infty}(u)\rangle \big)^{2} \\
&\leq &-\sum_{k=0}^{\infty}\lambda_{k}
\big(\int_{\Omega}\langle a_{k},P_{\infty}(u)\rangle \big)^{2} \\
&=&\mathcal{E}(P_{\infty}(u), P_{\infty}(u))\text{,}
\end{eqnarray*}
and, while $N\le N_0-1$,
\begin{eqnarray*}
-\int_{\Omega}\langle S\left(P_{N}(u)\right), P_{N}(u)\rangle
&=&-\sum_{k=0}^{N}\lambda_{k}\big(\int_{\Omega}\langle
a_{k},P_{\infty}(u)\rangle \big)^{2} \\
&\leq &-\sum_{k=0}^{\infty}\lambda_{k}
\big(\int_{\Omega}\langle a_{k},P_{\infty}(u)\rangle \big)^{2}
  +\sum_{k=N+1}^{N_0}\lambda_k\big(\int_{\Omega}\langle a_{k},P_{\infty}(u)\rangle \big)^{2}
 \\
&=&\mathcal{E}(P_{\infty}(u), P_{\infty}(u))+M\|u\|_2^2\text{.}
\end{eqnarray*}
This arrives the result expected.
\end{proof}

\begin{corollary}\label{coro-9-2}
For any $\varepsilon>0$, there exists $M(\varepsilon,\zeta)>0$
depending only on $\varepsilon$, $\zeta$, and $\Omega$ such that
\begin{equation}
\|(\nabla \times P_{N}(u), \nabla P_{N}(u))\|_{2}^{2}
 \leq M\|(\nabla\times u, u)\|_{2}^{2}   \label{00.42}
\end{equation}
for any $u\in H^{2}(\Omega)$ and any integer $N$.
\end{corollary}

\begin{proof}
Note that $P_{\infty}(u)\in K_{2}(\Omega)\cap H^{2}(\Omega)$ so that
\begin{equation*}
\nabla\times P_{\infty}(u)=\nabla\times P_{\infty}(u)=\nabla\times u
\end{equation*}
and the estimate follows from
\begin{equation*}
\mathcal{E}(P_{\infty}(u),P_{\infty}(u)) \leq M\|(\nabla
P_{\infty}(u), u)\|_{2}^{2},
\end{equation*}
which yields (\ref{00.42}). Similarly, the estimate for $\nabla
P_N(u)$ follows from (\ref{00.42}) and (\ref{est-june26-01a}).
\end{proof}

\section{Existence of Weak and Strong Solutions}

In this section, we consider the initial-boundary value problem
\eqref{23-n-1}--\eqref{nav-f},
where $\zeta>0$ is a constant.
The minimal requirement on the initial data is that $u_{0}\in
K_{2}(\Omega)$.

First we introduce the notion of weak solutions to the
initial-boundary problem \eqref{23-n-1}--\eqref{nav-f}.
We say that a vector field $u(t,x)$ is a weak solution of
\eqref{23-n-1}--\eqref{nav-f}
with slip length $\zeta$, provided that $u(t,x)$ satisfies the
following conditions:

\begin{enumerate}\renewcommand{\theenumi}{\roman{enumi}}
\item  For each $t>0$, $u(t,\cdot)\in K_{2}(\Omega)$, and $u\in
L^{2}([0,T]; H^{1}(\Omega))$ for any $T>0$;

\item  For any smooth vector field $\varphi(t,x)$
with $\varphi(t,\cdot)\in K_{2}(\Omega)$,
\begin{eqnarray}
&&\int_{\Omega}\langle u(T, \cdot), \varphi(T, \cdot)\rangle\notag\\
&&=\langle u_0, \varphi_{0}\rangle+\int_{0}^{T}\int_{\Omega}\langle
u(t,\cdot),\partial_t\varphi(t,\cdot)\rangle\notag\\
&&\,\,\,\, -\int_{0}^{T}\int_{\Omega}\langle\nabla\times u,
\left(u\times\varphi+\mu\nabla\times\varphi\right)\rangle
-\frac{\mu }{\zeta }\int_{0}^{T}\int_{\Gamma}\langle u,\varphi
\rangle +2\mu \int_{0}^{T}\int_{\Gamma}\pi(u,\varphi);\qquad
\label{n-3-03}
\end{eqnarray}

\item The energy inequality:
\begin{equation}
\|u(T,\cdot)\|_{2}^{2}+2\mu \int_{0}^{T}\|\nabla u\|_{2}^{2}
+2\mu\int_{0}^{T}\int_{\Gamma}\Big(\frac{1}{\zeta }|u|^{2}-\pi
(u,u)\Big) \leq \|u_{0}\|_{2}^{2}\text{.} \label{24-2-02}
\end{equation}
\end{enumerate}

Equation \eqref{n-3-03} is obtained by integrating (\ref{23-n-1})
and performing formally integration by parts.

\subsection{Construction of global weak solutions}

Notice that $(S, D_{\zeta}(S))$ has
the eigenvalues $\Lambda\geq\lambda_0\geq\lambda_{1}\geq
\cdots\geq\lambda_{n}\rightarrow -\infty$ and the eigenvector
functions $\{a_{n}\,:\, n=1,2,\cdots\}$, which form an orthonormal
basis of $K_{2}(\Omega)$. Then $Sa_{n}=\lambda_{n}a_{n}$, and
$a_{n}$ solves the Stokes equation:
\begin{equation}
\Delta a_{n}-\nabla p_{n}=\lambda _{n}a_{n}\text{,} \qquad
\nabla\cdot a_{n}=0, \label{2-0-1}
\end{equation}
subject to the kinematic condition \eqref{kinematic-1} and the
Navier's $\zeta$-condition \eqref{nav-f}.

Let
$$
u(t,\cdot)=\sum_{k=1}^{\infty}c_{k}(t)a_{k}\qquad\text{with}\,\,\,
c_{k}(t)=\int_{\Omega}\langle a_{k},u\rangle,
$$
be a solution of the
Navier-Stokes equation (\ref{23-n-1}) with initial data $u_{0}$.
Multiplying by $a_{k}$ and integrating over $\Omega$, we obtain
\begin{eqnarray*}
\partial_t c_{k} =\mu\int_{\Omega}\langle a_{k},\Delta
u\rangle -\int_{\Omega}\langle a_{k}, u\cdot \nabla u\rangle =\mu
\lambda_{k}c_{k}-\sum_{i,j=0}^{\infty}c_{i}c_{j}\int_{\Omega}\langle
a_{k},a_{i}\cdot \nabla a_{j}\rangle \text{.}
\end{eqnarray*}
Thus, for each integer $N$, we solve the Cauchy problem for the
system of differential equations:
\begin{eqnarray}
&&\frac{d}{dt}c_{k}=\mu
\lambda_{k}c_{k}-\sum_{i,j=0}^{N}c_{i}c_{j}\int_{\Omega}\langle
a_{k},a_{i}\cdot\nabla a_{j}\rangle  \label{ord-30-01}\\
&&c_k|_{t=0}=\int_{\Omega}\langle a_k, u_{0}\rangle.
\end{eqnarray}

Define
$$
u^{N}(t,\cdot)=\sum_{k=0}^{N}c_{k}(t)a_{k}.
$$
Then $u^N(t,\cdot)\in D_{0,\zeta}(S)$ for $t>0$ and satisfies the
evolution equation:
\begin{equation}
\partial_t u^N
=\mu S(u^N)-\sum_{k=1}^{N}a_{k}\int_{\Omega}\langle a_{k},
u^N\cdot\nabla u^N\rangle. \label{ord-30-02}
\end{equation}
Therefore, $\nabla\cdot u^N=0$, $\left.
(u^N)^{\bot}\right|_{\Gamma}=0$, and $\left. (\nabla \times
u^N)^{\Vert}\right|_{\Gamma}=-\frac{1}{\zeta}(\ast u^N)+2\big(\ast
\pi(u^N)\big)$ for $t>0$.

Now we make the energy estimates for $u^N$ and $u^N_{t}$. First, it
is easy to see that
\begin{eqnarray*}
\frac{d}{dt}\|u^N\|_{2}^{2} &=&2\sum_{k=1}^{N}c_{k}\frac{d}{dt}c_{k}
=2\mu \sum_{k=1}^{N}\lambda_{k}c_{k}^{2}-2\int_{\Omega}\langle u^N,
u^N\cdot\nabla u^N\rangle \text{.}
\end{eqnarray*}
Since
\begin{eqnarray*}
\sum_{k=1}^{N}\lambda_{k}c_{k}^{2} =\int_{\Omega}\langle
S(u^N),u^N\rangle =-\int_{\Omega}|\nabla
u^N|^{2}-\frac{1}{\zeta}\int_{\Gamma}|u^N|^{2}+\int_{\Gamma}\pi(u^N,u^N),
\end{eqnarray*}
we have
\begin{eqnarray*}
\frac{d}{dt}\|u^N\|_{2}^{2} =-2\mu \int_{\Omega}|\nabla
u^N|^{2}-\int_{\Omega}u^N\cdot \nabla (|u^N|^{2})
-\frac{2\mu}{\zeta}\int_{\Gamma}|u^N|^{2}+2\mu \int_{\Gamma}\pi(u^N,
u^N).
\end{eqnarray*}
Since $\nabla\cdot u^N=0$ and $(u^N)^{\bot}=0$, which implies
$\int_{\Omega}u^N\cdot\nabla(|u^N|^{2})=0$, we obtain the energy
balance identity:
\begin{equation}
\frac{d}{dt}\|u^N\|_{2}^{2}+2\mu \|\nabla
u^N\|_{2}^{2}=-\frac{2\mu}{\zeta} \|u^N\|_{L^{2}(\Gamma)}^{2}
+2\mu\int_{\Gamma}\pi(u^N,u^N). \label{en-01}
\end{equation}
Therefore, we have
\begin{equation}
\|u^N(T,\cdot)\|_{2}^{2}+2\mu \int_{0}^{T}\|\nabla u^N\|_{2}^{2}
+2\mu\int_{0}^{T}\int_{\Gamma}\Big(\frac{1}{\zeta}|u^N|^{2}
-\pi(u^N,u^N)\Big) =\|u_{0}^N\|_{2}^{2}\le \|u_0\|_2^2.
\label{24-2-01}
\end{equation}

Using the Sobolev embedding inequality:
$$
\int_{\Gamma} |u|^2 \le \epsilon \|\nabla u\|^2_2 + C(\epsilon)
\|u\|^2_2,
$$
we have
$$
\int_\Gamma\pi(u^N,u^N) \le \frac{1}{2}\|\nabla u^N\|^2_2
+C(\varepsilon)\|u^N\|^2_2.
$$
Then, from \eqref{24-2-01}, we have
\begin{equation}\label{energy-est}
\|u^N(T,\cdot)\|_{2}^{2}+\mu \int_{0}^{T}\|\nabla u^N\|_{2}^{2}
+\frac{2\mu}{\zeta}\int_{0}^{T}\int_{\Gamma}|u^N|^{2} \le
\|u_{0}\|_{2}^{2}+C\int_0^T\|u^N(s,\cdot)\|_2^2.
\end{equation}
The Gronwall inequality and \eqref{energy-est} imply that
$\|u^{N}(t,\cdot)\|_{2}^{2}$ and $ \int_{0}^{t}\|\nabla
u^{N}\|_{2}^{2}$ are uniformly bounded in $t$, $\zeta$, and $N$. The
\emph{apriori} estimate (\ref{24-2-01}) also ensures that, for each
integer $N$, system (\ref{ord-30-01}) has a unique solution for all
$t>0$. Then we conclude

\begin{theorem}\label{w-th-01} Let $u_0\in K_2(\Omega)$. Then, for any $T>0$,
the family $\{u^{N}(t,x)\}, 0\le t\leq T$, is weakly compact in the
space $L^{2}([0,T]; K_{2}(\Omega))$ so that it has a convergent
subsequence that converges to a vector field $u\in L^2([0, T];
K_2(\Omega))$, and the limit function $u(t,x)$ is a weak solution to
problem \eqref{23-n-1}--\eqref{nav-f}.
\end{theorem}

\medskip
Furthermore, we have
\begin{equation}
\|u_{\zeta}(T,\cdot)\|_{2}^{2}+\mu \int_{0}^{T}\int_{\Omega}\|\nabla
u_{\zeta}(t,\cdot)\|_{2}^{2}+\frac{2\mu}{\zeta}\int_0^T\int_\Gamma|u_\zeta|^2\leq
\|u_{0}\|_{2}^{2}\text{,} \label{24-2-04}
\end{equation}
where we have used $u_{\zeta}$ to indicate the dependence on the
slip length $\zeta>0$. The uniform energy estimate (\ref{24-2-04})
implies that the family $u_{\zeta}(t,x)$
is pre-compact in $L^{2}([0,T]; K_{2}(\Omega))$ and hence there
exists a convergent subsequence (still denoted)
$u_{\zeta}\rightarrow u$, when $\zeta\to 0$. From \eqref{24-2-04},
we have
$$
\int_0^T\int_\Gamma|u_\zeta|^2\le \frac{\zeta}{2\mu}\|u_0\|^2_2,
$$
which leads to
\begin{equation*}
\int_{0}^{T}\int_{\Gamma}|u|^{2}=0.
\end{equation*}
This implies that the limit $u(t,\cdot)$ is subject to the no-slip
condition for almost all time $t$. On the other hand, when $\zeta\to
\infty$, we obtain that there also exists a subsequence (still
denoted) $u_\zeta(t,x)$ converging to $u(t,x)$ in $L^2([0,T];
K_2(\Omega))$ such that $u(t,x)$ is a solution to \eqref{23-n-1}
subject to the complete slip boundary condition:
\begin{equation}\label{complete-slip}
\left. \omega^{\Vert}\right|_{\Gamma} =2\big(\ast\pi(u)\big)
\end{equation}
in the weak sense.

\begin{theorem}\label{w-th-02} Let $u_\zeta(t,x),\, 0\le t\leq T$, be
a weak solution to problem \eqref{23-n-1}--\eqref{nav-f} constructed
in Theorem {\rm \ref{w-th-01}}. Then

{\rm (i)} when $\zeta\to 0$, there exists a subsequence (still
denoted) $u_\zeta(t,x)$ converging to $u(t,x)$ in $L^2([0,T];
K_2(\Omega))$ such that $u(t,x)$ is a solution to \eqref{23-n-1}
subject to the no-slip condition for almost all time $t$;

{\rm (ii)} when $\zeta\to \infty$, there  exists a subsequence
(still denoted) $u_\zeta(t,x)$ converging to $u(t,x)$ in $L^2([0,T];
K_2(\Omega))$ such that $u(t,x)$ is a solution to \eqref{23-n-1}
subject to the complete slip boundary condition
\eqref{complete-slip}.
\end{theorem}

The nonhomogeneous vorticity boundary problem related to
\eqref{complete-slip} has been investigated in \cite{C-Q1}.

\subsection{Strong solutions}

In this section, we prove that there exists a strong solution to
problem \eqref{23-n-1}--\eqref{nav-f} for small time. To this end,
we develop the $L^{2}$-estimates for $u^{N}$ up to second-order
derivatives, uniformly in $N$. More precisely, we prove the
following:

\begin{theorem}\label{m-th-24-03} Let $u_0\in K_2(\Omega)\cap
H^2(\Omega)$. Then there exist $T^{\ast}>0$ and $M>0$ depending only
on $\zeta$, $\varepsilon$, $\mu$, $\Omega$, and $\|u_{0}\|_{H^{2}}$
(but independent of $N$) such that
\begin{equation}
\|u^{N}(t,\cdot)\|_{H^{2}}^{2}+\|\partial_{t}u^{N}(t,\cdot)\|_{2}^{2}\leq
M\text{.} \label{e-25-1}
\end{equation}
\end{theorem}

\begin{proof}
For simplicity, we write $u=u^{N}$ given by (\ref{ord-30-01}).
Let $\omega =\nabla\times u$ and
$\psi=\nabla\times \omega =-\Delta u$ as usual. Recall that $u$
fulfils the evolution equation:
\begin{equation*}
\partial_t u=\mu S(u)-\sum_{k=1}^{N}a_{k}\int_{\Omega}\langle
a_{k},u\cdot \nabla u\rangle,
\end{equation*}
where, for $t>0$, $u(t,\cdot)\in D_{0,\zeta}(S)$. Taking the
$t$-derivative  yields the evolution equation:
\begin{equation}
\partial_tu_{t}
=\mu S(u_{t}) -\sum_{k=1}^{N}a_{k}\int_{\Omega}\langle
a_{k},u_{t}\cdot\nabla u\rangle
-\sum_{k=1}^{N}a_{k}\int_{\Omega}\langle a_{k},u\cdot\nabla
u_{t}\rangle,
\label{ord-30-03}
\end{equation}
and $u_{t}(t,\cdot)\in D_{0,\zeta}(S)$. Therefore, $\nabla\cdot
u_{t}=0$ and $u_{t}$ again satisfies the same boundary conditions as
those of $u$. Using the evolution equation (\ref{ord-30-03}), we
have
\begin{eqnarray}
\frac{d}{dt}\|u_{t}\|_{2}^{2} &=&2\mu \int_{\Omega}\langle S(u_{t}),
u_{t}\rangle -2\int_{\Omega}\langle u_{t},u_{t}\cdot\nabla u\rangle
-2\int_{\Omega}\langle u_{t}, u\cdot\nabla u_{t}\rangle \nonumber\\
&=&2\mu \int_{\Omega}\langle S(u_{t}),u_{t}\rangle
-2\int_{\Omega}\langle u_{t}, u_{t}\cdot\nabla u\rangle \text{.}
\label{5.11a}
\end{eqnarray}
Integration by parts yields
\begin{eqnarray*}
\int_{\Omega}\langle \Delta u_{t},u_{t}\rangle &=&-\|\nabla
u_{t}\|_{2}^{2}+\int_{\Gamma}\langle u_{t}\times \left(\nabla \times
u_{t}\right), \nu \rangle +\int_{\Gamma}\langle u_{t}\cdot\nabla
u_{t},\nu\rangle \\
&=&-\|\nabla u_{t}\|_{2}^{2}-\frac{1}{\zeta}\int_{\Gamma}\langle
u_{t}\times(\ast u_{t}),\nu \rangle +2\int_{\Gamma}\langle
u_{t}\times(\ast \pi(u_{t})), \nu\rangle
-\int_{\Gamma}\pi(u_{t},u_{t}) \\
&=&-\|\nabla
u_{t}\|_{2}^{2}-\frac{1}{\zeta}\int_{\Gamma}|u_{t}|^{2}+\int_{\Gamma}
\pi(u_{t},u_{t})\text{.}
\end{eqnarray*}
Substitution this into \eqref{5.11a} leads to
\begin{eqnarray}
\frac{d}{dt}\|u_{t}\|_{2}^{2}=-2\mu \|\nabla
u_{t}\|_{2}^{2}-4\int_{\Omega}\langle u_{t},u_{t}\cdot\nabla
u\rangle-\frac{2\mu}{\zeta}\int_{\Gamma}|u_{t}|^{2}
+2\nu\int_{\Gamma}\pi(u_{t},u_{t})\text{.} \label{time-01}
\end{eqnarray}
It follows that
\begin{eqnarray}
\frac{d}{dt}\|u_{t}\|_{2}^{2} &\leq &-\frac{2\mu}{\zeta}
\|u_{t}\|_{L^{2}(\Gamma)}^{2}-2\mu\|\nabla
u_{t}\|_{2}^{2}+4\|u_{t}\|_{2}^{2}\|\nabla u\|_{\infty}
+2\mu\int_{\Gamma}\pi(u_{t},u_{t})  \notag \\
&\leq &-\frac{2\mu}{\zeta}\|u_{t}\|_{L^{2}(\Gamma)}^{2} -\mu\|\nabla
u_{t}\|_{2}^{2}+\varepsilon \|\nabla^3u\|^2
+C\left(\|(\psi, u,
u_{t})\|_{2}^{2}+\|u_{t}\|_{2}^{4}\right),\qquad \label{e-25-02}
\end{eqnarray}
where the Sobolev's imbedding has been used and $C>0$ is a constant
depending only on the domain $\Omega$.

Next we deal with $\Delta u$. The evolution equation for
$u(t,\cdot)$ may be written as
\begin{equation*}
\partial_t u=\mu S(u)-P_{N}\left( u\cdot \nabla u\right).
\end{equation*}
Together with the vector identity $u\cdot \nabla u=\frac{1}{2}\nabla
|u|^{2}-u\times\omega$, we have
\begin{equation}
\partial_t u=\mu S(u)+P_{N}(u\times\omega)\text{. }
\label{9-0-11}
\end{equation}
Since $\nabla \times S(u)=\nabla \times (\Delta u)$, taking curl
(twice) to both sides of equation (\ref{9-0-11}) yields
\begin{equation}
\partial_t\omega =\mu \Delta\omega
+\nabla\times P_{N}(u\times\omega), \label{vo-30-01}
\end{equation}
and
\begin{equation}
\partial_t\psi =\mu \Delta \psi
+\nabla\times\nabla\times P_{N}\left(u\times\omega\right) \text{.}
\label{vo-30-02}
\end{equation}

It follows from (\ref{vo-30-02}) that
\begin{equation}
\frac{d}{dt}\|\psi\|_{2}^{2} =2\mu\int_{\Omega}\langle\Delta\psi,
\psi\rangle +2\int_{\Omega}\langle\nabla\times\nabla\times
P_{N}\left(u\times\omega \right), \psi \rangle\text{.}
\label{6-0-1}
\end{equation}
Integration by parts leads to
\begin{equation*}
2\mu\int_{\Omega}\langle\Delta\psi, \psi\rangle
=-2\mu\|\nabla\psi\|_{2}^{2} +\mu\int_{\Gamma}
\partial_\nu(|\psi|^{2}),
\end{equation*}
and
\begin{eqnarray*}
&&\int_{\Omega}\langle\nabla\times\nabla\times P_{N}
\left(u\times\omega\right), \psi \rangle\\
&&=\int_{\Omega}\langle\nabla\times P_{N}\left(u\times\omega\right),
\nabla\times\psi\rangle -\int_{\Gamma}\langle\psi\times
\left(\nabla\times P_{N}\left(u\times\omega\right)\right), \nu\rangle  \\
&&=\int_{\Omega}\langle\nabla\times P_{N}\left(u\times\omega\right),
  \nabla \times\psi\rangle
  -\int_{\Gamma }\langle \psi \times \big(-\frac{1}{\zeta }(\ast
P_{N}\left( u\times \omega \right)) +2(\ast \pi \left( P_{N}\left(
u\times\omega \right)\right))\big), \nu \rangle  \\
&&=\int_{\Omega }\langle \nabla \times P_{N}\left(u\times\omega
\right),\nabla \times \psi\rangle
+\frac{1}{\zeta}\int_{\Gamma}\langle \psi,
P_{N}\left(u\times\omega\right)\rangle-2\int_{\Gamma}\pi\left(\psi,
P_{N}\left(u\times\omega\right)\right)\text{.}
\end{eqnarray*}
Therefore, we obtain
\begin{eqnarray}
\frac{d}{dt}\|\psi\|_{2}^{2}&=&-2\mu\|\nabla\psi\|_{2}^{2}
+2\int_{\Omega}\langle \nabla\times P_{N}\left(u\times\omega\right),
\nabla \times \psi \rangle +\mu\int_{\Gamma}\partial_\nu\big(|\psi|^{2}\big)
  \notag \\
&&+\frac{2}{\zeta}\int_{\Gamma}\langle\psi,
P_{N}\left(u\times\omega\right)\rangle -4\int_{\Gamma}\pi\left(\psi,
P_{N}\left(u\times\omega\right)\right) \text{.} \label{6-031}
\end{eqnarray}
Using the H\"{o}lder inequality, one obtains
\begin{eqnarray}
\frac{d}{dt}\|\psi\|_{2}^{2} &\leq &-2\mu
\|\nabla\psi\|_{2}^{2}+2\|\nabla\times
P_{N}\left(u\times\omega\right)\|_{2}\|\nabla\times\psi\|_{2}  \notag \\
&&+\mu
\int_{\Gamma}\partial_\nu\big(|\psi|^{2}\big)+M\|\psi\|_{L^{2}(\Gamma)}
\|P_{N}\left(u\times\omega\right)\|_{L^{2}(\Gamma)} \text{.}
\label{23-f-01}
\end{eqnarray}
The first boundary integral
$\int_{\Gamma}\partial_\nu\big(|\psi|^{2}\big)$ can be estimated by
using Lemma \ref{le-23-4} to obtain
\begin{equation*}
\frac{1}{2}\int_{\Gamma}\partial_\nu\big(|\psi|^{2}\big)\leq
\varepsilon \|\nabla^3 u\|_{2}^{2}+M\|(\psi, u)\|_{2}^{2}\text{.}
\end{equation*}
The product of last two boundary integrals in (\ref{23-f-01}) can be
estimated via the Sobolev imbedding to yield
\begin{eqnarray*}
\quad
M\|\psi\|_{L^{2}(\Gamma)}\|P_{N}\left(u\times\omega\right)\|_{L^{2}(\Gamma)}
\leq \big(\varepsilon\|\nabla\psi\|_{2}
  +M\|\psi\|_{2}\big) \big(\varepsilon\|\nabla P_{N}\left(u\times\omega\right)\|_{2}
  +M\|u\times \omega\|_{2}\big).
\end{eqnarray*}
Plugging these estimates into (\ref{23-f-01}), using Corollary
\ref{coro-9-2} and the estimate
\begin{eqnarray*}
\|u\times\omega\|_{2}&\leq & M\|u\|_{H^{1}}^{2}=M\|(\psi,
u)\|_{2}^{2},
\end{eqnarray*}
and rearranging the inequality, we obtain
\begin{eqnarray}
\frac{d}{dt}\|\psi\|_{2}^{2}
&\leq&-2\mu\|\nabla\psi\|_{2}^{2}+\varepsilon\|\psi\|_{H^{1}}^{2}+\varepsilon
\|\nabla^3 u\|_{2}^{2} +M\|\nabla
P_{N}(u\times\omega)\|_{2}\|\psi\|_{H^{1}}
\notag \\
&&+M\left( \|(\psi, u)\|_{2}^{2}+\|(\psi,u)\|_{2}^{4}\right).
\label{23-f-02}
\end{eqnarray}
Finally, we use Corollary \ref{coro-9-2} to obtain
\begin{eqnarray*}
\|\nabla P_{N}\left(u\times \omega \right)\|_{2} &\leq
&M\|\nabla\times\left(u\times\omega\right)\|_{2}
+M\|u\times\omega\|_{2}\\
&\leq &M\|\big(\omega\cdot\nabla u,\, u\cdot\nabla\omega,\,
u\times\omega\big)\|_{2} \leq M\|u\|_{H^{2}}^{2}.
\end{eqnarray*}
Then we conclude
\begin{equation}
\frac{d}{dt}\|\psi\|_{2}^{2}\leq -\frac{3}{2}\mu
\|\nabla\psi\|_{2}^{2}+M\big(\|(\psi, u)\|_{2}^{2} +\|(\psi,
u)\|_{2}^{4}\big) \label{23-5-01}
\end{equation}
for some constant $M$ depending only on $\zeta, \mu$, and $\Omega$.

Let $F=\|(\psi, u, u_{t})\|_{2}^{2}$. Combining (\ref {en-01}) and
(\ref{e-25-02}) with (\ref{23-5-01}), we obtain the following
differential inequality:
\begin{eqnarray}
\frac{d}{dt}F \leq -\mu \|\nabla(\psi, u,
u_{t})\|_{2}^{2}-\frac{2\mu}{\zeta}\|(u,
u_{t})\|_{L^{2}(\Gamma)}^{2}+M_{1}F+M_{2}F^{2} \label{e3-05}
\end{eqnarray}
for some constants $M_1$ and $M_2$ depending only on $\zeta$ and
$\Omega$. In particular, we have
\begin{equation}
\frac{d}{dt}F\leq M_{1}F+M_{2}F^{2}\text{.}  \label{e3-06}
\end{equation}
Since
\begin{eqnarray*}
\|u_{t}\|_{2} \leq \mu \|S(u)\|_{2}+\|P_{N}\left( u\times \omega
\right)\|_{2} \leq 2\mu\|\Delta u\|_{2}+\|u\times\omega\|_{2}\leq
2\mu\|\Delta u\|_{2}+\|u\|_{2}\|\omega\|_{2},
\end{eqnarray*}
then
$$
F(0)\leq C\left(\mu,\Omega\right) \|u_{0}\|_{H^{2}}^{2}
$$
for some constant $C\left(\mu, \Omega\right)$ depending only on
$\mu$ and $\Omega$.

Let $\rho$ be the solution on $[0,T^{\ast})$ to the ordinary
differential equation:
\begin{equation}
\rho^{\prime}=M_{1}\rho +M_{2}\rho^{2}\text{ , \ \ \ } \rho
(0)=C\left(\mu, \Omega\right)\|u_{0}\|_{H^{2}}^{2}, \label{e3-07}
\end{equation}
where $T^{\ast}>0$ is the blowup time of $\rho$.

Then inequality (\ref{e3-06}) together with the fact that
$F(0)\leq\rho(0)$ implies that $F(t)\leq \rho(t)$ on $[0,T^{\ast})$.
This completes the proof of Theorem \ref{m-th-24-03}.
\end{proof}

\begin{theorem}
Let $u_{0}\in K_{2}(\Omega)\cap H^{2}(\Omega)$. Then there exists
$T^{\ast}>0$ depending only on $\zeta, \mu$, $\Omega$, and
$\|u_{0}\|_{H^{2}}$ such that there is a strong solution $u(t,x)$ of
the initial-boundary value problem \eqref{23-n-1}--\eqref{nav-f} up
to $T^{\ast}>0$.
\end{theorem}

\section{Inviscid limit as $\mu\to 0$}

In this section, we  analyze the inviscid limit of the solutions
$u^\mu(t,x)$ of the initial-boundary value problem
\eqref{23-n-1}--\eqref{nav-f}.

Let $u(t,x)$ be the unique smooth solution of the initial-boundary
value problem of the Euler equations:
\begin{equation}
\left\{
\begin{array}{ll}
{\partial_t}u+u\cdot \nabla u=-\nabla p\text{ , \ \ \ \ \ \ \ \ }
\\
\nabla\cdot u=0,
\\
u(0,\cdot )=u(0),\\
\left. u^{\bot}\right|_{\Gamma}=0,
\end{array}
\right.  \label{eu-01}
\end{equation}
up to time $T^{\ast}>0$. Notice that all solutions $u^{\mu}$ to
problem \eqref{23-n-1}--\eqref{nav-f}, $\mu\in (0, \mu_0]$ for some
$\mu_0>0$, subject to the same boundary conditions: $\left.
(u^{\mu})^{\bot}\right|_{\Gamma}=0$ and
$$
\left. (\nabla\times
u^{\mu})^{\Vert}\right|_{\Gamma}=-\frac{1}{\zeta}(*u^\mu)+2\big(*\pi(u^\mu)\big),
$$
while the solution $u$ of \eqref{eu-01} satisfies only the kinematic
boundary condition
and is
independent of the viscosity constant $\mu$.

\begin{theorem}
Suppose that, for all $\mu\in (0, \mu_{0}]$, a unique strong
solution $u^{\mu}$ of the initial-boundary value problem
\eqref{23-n-1}--\eqref{nav-f} and  a unique strong solution $u\in
H^2(\Omega)$ to the initial-boundary value problem \eqref{eu-01})
both exist up to time $T^{\ast}>0$. Then there exists $C=C(\mu_0, T,
\|u\|_{L^2([0,T]; H^2\cap W^{1,\infty}(\Omega)})$, independent of
$\mu$, such that, for any $T\in [0, T^{\ast}]$,
\begin{equation}
\sup_{0\le t\leq T}\|u^{\mu}(t,\cdot )-u(t,\cdot)\|_{2}\le
C\,\mu\rightarrow 0\text{ \ \ as }\mu \downarrow 0, \label{th01}
\end{equation}
and
\begin{equation*}
\int_{0}^{T}\|\nabla \left( u^{\mu}-u\right)
(s,\cdot)\|_{2}^{2}ds\leq C.
\end{equation*}
 It follows that the whole solution sequence $u^\mu$
of \eqref{23-n-1}--\eqref{nav-f} converges to the unique solution
$u(t,x)$ of the initial-boundary value problem \eqref{eu-01} in
$L^2$ as $\nu\to 0$.
\end{theorem}

\begin{proof}
Let $v^{\mu}=u^{\mu}-u$. Then $v^{\mu}$ satisfies the following
equations:
\begin{equation}
\left\{
\begin{array}{ll}
\partial_t v^{\mu} =\mu \Delta
v^{\mu}-\left(v^{\mu}+u\right)\cdot\nabla v^{\mu} -\nabla
P^{\mu}-v^{\mu}\cdot\nabla u+\mu\Delta u, \\
\nabla\cdot v^{\mu}=0\text{,}
\end{array}
\right.  \label{0-c-1}
\end{equation}
and the initial condition:
\begin{equation}\label{initial}
v^{\mu}(0,\cdot)=0,
\end{equation}
where $P^{\mu}=p^{\mu}-p$. Since both $u^{\mu}$ and $u$ satisfy the
kinematic condition \eqref{kinematic-1}, so does $v^{\mu}$. Thus, by
means of the energy method, we obtain
\begin{eqnarray*}
\frac{d}{dt}\|v^{\mu}\|_{2}^{2} &=&2\mu \int_{\Omega}\langle
v^{\mu}, \Delta v^{\mu}\rangle -\int_{\Omega}\langle v^{\mu}+u,
 \nabla(|v^{\mu}|^{2})\rangle
-2\int_{\Omega}\langle \nabla P^{\mu}, v^{\mu}\rangle \\
&&-2\int_{\Omega}\langle v^{\mu}\cdot\nabla u, v^{\mu}\rangle
+2\mu\int_{\Omega}\langle \Delta u, v^{\mu}\rangle.
\end{eqnarray*}
Integration by parts in the first three integrals leads to
\begin{eqnarray*}
\frac{d}{dt}\|v^{\mu}\|_{2}^{2} &=&-2\mu \int_{\Omega}|\nabla\times
v^{\mu}|^{2}+2\mu\int_{\Gamma}\langle v^{\mu}\times
\left(\nabla\times v^{\mu}\right), \nu\rangle \\
&&-2\int_{\Omega}\langle v^{\mu}\cdot\nabla u, v^{\mu}\rangle
+2\mu\int_{\Omega}\langle\Delta u, v^{\mu}\rangle \\
&=&-2\mu\|\nabla v^{\mu}\|_{2}^{2}-2\mu \int_{\Gamma}\pi(v^{\mu},
v^{\mu})
+2\mu \int_{\Gamma}\langle v^{\mu}\times b, \nu\rangle \\
&&-2\int_{\Omega}\langle v^{\mu}\cdot\nabla u, v^{\mu}\rangle
+2\mu\int_{\Omega}\langle \Delta u, v^{\mu}\rangle,
\end{eqnarray*}
where
$b=-\frac{1}{\zeta}\big(*(v^\mu+u)\big)+2\big(*\pi(v^\mu+u)\big)-\big(\nabla\times
u\big)^\Vert$.

Furthermore, we  use the following estimate:
\begin{eqnarray*}
\int_{\Gamma}\langle v_{\mu}\times b,\nu \rangle \leq
\|b\|_{L^{2}(\Gamma)}\|v_{\mu}\|_{L^{2}(\Gamma)} \le
C\big(\|u\|_{H^1(\Gamma)}^2 +\|v^{\mu}\|_{L^{2}(\Gamma)}^{2}\big)
\end{eqnarray*}
to obtain
\begin{eqnarray}
\frac{d}{dt}\|v^{\mu}\|_{2}^{2} &\leq& -2\mu\|\nabla
v^{\mu}\|_{2}^{2}+\mu C\big( \|v^{\mu}\|_{L^{2}(\Gamma)}^{2}
+\|u\|_{H^{1}(\Gamma)}^{2}\big)  \notag \\
&&+2\|\nabla u\|_{\infty}\|v^{\mu}\|_{2}^{2}+2\mu \|\Delta
u\|_{2}\|v^{\mu}\|_{2}\text{.}  \label{p-01}
\end{eqnarray}
Finally, we use the Sobolev imbeddings:
\begin{equation*}
C\|v^{\mu}\|_{L^{2}(\Gamma)}^{2}\leq \|\nabla v_{\mu}\|_{2}^{2}+
\tilde{C}\|v_{\mu}\|_{2}^{2}, \qquad \|u\|_{H^1(\Gamma)}^2\le
C\|u\|_{H^2(\Omega)}^2
\end{equation*}
to establish the differential inequality:
\begin{eqnarray}
\frac{d}{dt}\|v_{\mu}\|_{2}^{2}+\mu \|\nabla v_{\mu}\|_{2}^{2}\le
C\left(\|\nabla u\|_{\infty}+\mu_0\right) \|v_{\mu}\|_{2}^{2}
+\mu \|u\|_{H^{2}(\Omega)}^{2}\text{.} \label{p-02}
\end{eqnarray}
Gronwall's inequality implies that
\begin{eqnarray}
\|v^{\mu}(t, \cdot)\|_{2}^{2} \leq &
\mu \int_{0}^{t}e^{C(\int_{s}^{t}\|\nabla
u(\tau,\cdot)\|_{\infty}d\tau+\mu_0
(t-s))}\|u(s,\cdot)\|_{H^{2}(\Omega)}^{2} ds=:C\mu \text{,}
\label{p-04}
\end{eqnarray}
where $C>0$ depends only on $\mu_0, T^*$, and $\|u\|_{L^2([0,T];
H^2\cap W^{1,\infty}(\Omega)}$. Using this and \eqref{p-02}, we
further have
\begin{equation*}
\int_{0}^{t}\|\nabla v_{\mu}(s, \cdot)\|_{2}^{2}ds\leq C\mu.
\end{equation*}
This completes the proof.
\end{proof}

In order to ensure the convergence of $u^{\mu}$ to $u$ in the strong
sense (say, $H^{2}(\Omega)$), a necessary condition is that $u$ must
match the Navier's $\zeta$-condition \eqref{nav-f}.

\medskip
\bigskip
{\bf Acknowledgments.} Gui-Qiang Chen's research was supported in
part by the National Science Foundation under Grants DMS-0807551,
DMS-0720925, and DMS-0505473, and the Natural Science Foundation of
China under Grant NSFC-10728101. Zhingmin Qian's research was
supported in part by EPSRC grant EP/F029578/1. This paper was
written as part of  the international research program on Nonlinear
Partial Differential Equations at the Centre for Advanced Study at
the Norwegian Academy of Science and Letters in Oslo during the
academic year 2008--09.

\end{document}